\documentclass{article}
\usepackage[utf8]{inputenc}
\usepackage{geometry}
\usepackage{graphicx}
\usepackage{epsfig}
\usepackage{epstopdf}
\usepackage{hyperref}
\hypersetup{
    colorlinks=true,
    linkcolor=red,
    citecolor = blue,
    filecolor=magenta,      
    urlcolor=cyan,
    pdfpagemode=FullScreen,
    }

\usepackage{comment}
\includecomment{workcomment}
\specialcomment{workcomment}{\begingroup\color{red}}{\endgroup}

\usepackage{times,amsmath}
\usepackage{latexsym,amssymb,lscape,multirow}
\let\geq=\geqslant
\let\leq=\leqslant

\usepackage{dcolumn,amsthm}

\usepackage{lastpage}
\usepackage{mathptmx}
\usepackage{blindtext}
\include{graphics}
\usepackage{amsthm}
\usepackage{multirow}
\usepackage{xcolor}
\usepackage{appendix}

\AtBeginDocument{
	\label{CorrectFirstPageLabel}
	
}

\usepackage[export]{adjustbox}

\newtheorem{thm}{Theorem}[section]
\newtheorem{lm}[thm]{Lemma}

\newtheorem{pro}[thm]{Proposition}
\newtheorem{df}{Definition}[section]

\numberwithin{equation}{section}
\newtheorem*{remark}{Remark}

\renewcommand\Re{\operatorname{Re}}

\newcommand{\RR}{\mathbb{R}}

\usepackage{authblk}

\begin{document}


\title{A stiffly stable semi-discrete scheme for the damped wave equation on the half-line using SBP and SAT techniques}

\author{Thi Hoai Thuong NGUYEN}
\affil{Faculty of Mathematics and Computer Science, University of Science, 227, Nguyen Van Cu, 700000, Ho Chi Minh City, Vietnam}
\affil{Vietnam National University, 700000, Ho Chi Minh City, Vietnam.}
%
\author{Benjamin BOUTIN}
\affil{IRMAR (UMR CNRS 6625), Universit\'e de Rennes, Campus de Beaulieu, 35042 Rennes Cedex, France.}

\date{\today}

%


\maketitle

\begin{abstract}
	This paper investigates the stability of both the semi-discrete and the implicit central scheme for the linear damped wave equation on the half-line, where the spatial boundary is characteristic for the limiting equation. The proposed schemes incorporate a discrete boundary condition designed to guarantee the uniform stability of the IBVP, regardless of the stiffness of the source term or the spatial step size. Stability estimates for the semi-discrete scheme are established using the summation-by-parts (SBP) and simultaneous-approximation-term (SAT) penalty techniques, building on the continuous framework analyzed by Xin and Xu (2000) \cite{XinXu00}.
\end{abstract}


\section{Introduction}

\subsection{Context and motivation}

Hyperbolic systems of partial differential equations with relaxation terms are relevant in a variety of physical applications. Reactive flows \cite{ColellaMajdaRoytburd86}, water waves \cite{Whitham74,Stoker92}, and relaxing gas theory \cite{Clarke78} are examples of such systems. Following the research of Liu \cite{Liu87}, Hanouzet and Natalini \cite{HanouzetNatalini03}, and Yong \cite{Yong99, Yong04}, the investigation of the zero relaxation limit for such systems has attracted significant interest from both a theoretical and numerical perspective.

We are concerned with the numerical treatment of boundaries for hyperbolic relaxation systems. The simplest linear hyperbolic system with relaxation is the following linear damped wave equation:
\begin{subequations}
	\label{Problem}
	\begin{align}\label{I_1}
	\begin{split}
		\begin{cases}
			\partial_t u^\varepsilon +\partial_x v^\varepsilon =0,
			\\
			\partial_t v^\varepsilon +a\partial_x u^\varepsilon = -\varepsilon^{-1}v^\varepsilon,
		\end{cases}
	\end{split}
\end{align}
where $u^\varepsilon, v^\varepsilon\in\mathbb{R}$ and $a>0$.
The relaxation parameter $\varepsilon>0$  corresponds to the typical time in the process of return to the equilibrium.
Actually, the first order (in $\varepsilon$) equilibrium system is $v=0$ together with the stationary evolution equation $\partial_t u = 0$ driving the remaining physical quantity $u$. In this respect, the equilibrium limit is doubtlessly of limited interest. Let us highlight that we are more interested in the limiting process that in the limit itself. The system can be understood as an archetype of more general situations that may involve more complex equilibrium evolution processes together with, as we will discuss afterwards, characteristic boundaries in the limiting model. The subject concerns the initial boundary value problem (IBVP) for \eqref{I_1} in the quarter plane~$x>0,\ t>0$. This problem is supplemented with some initial data at time~$t=0$:
\begin{align}\label{I_2}
	u^\varepsilon(x,0)=u_0(x),\quad
	v^\varepsilon(x,0)=v_0(x).
\end{align}
The \textbf{hyperbolic structure} of the first order terms in the left-hand side of~\eqref{I_1} requires a condition on the solution at the spatial boundary $x=0$. We assume the boundary condition to be linear and of the form
\begin{align}\label{I_3}
	B_uu^\varepsilon(0,t) + B_vv^\varepsilon(0,t)=b(t),
\end{align}
\end{subequations}%
\color{black}%
\stepcounter{equation}%
where $B_u$ and $B_v$ are given real constants and the function $b$ is some boundary data. For a given general hyperbolic problem, no specific boundary condition is inherently preferable for defining a well-posed IBVP, except that the boundary condition must satisfy some structural assumptions. In fact, the parameters $(B_u,B_v)$ and data $b(t)$ in \eqref{I_3} are usually determined from physical modelling considerations. They should satisfy certain algebraic conditions which are recalled hereafter.
The problem \eqref{I_1} represents a particular simple instance of the Jin-Xin relaxation model in one spatial dimension \cite{JinXin95}. Its hyperbolic structure is associated with the Riemann invariants $\sqrt{a}u^\varepsilon\pm v^\varepsilon$ and the corresponding characteristic velocities $\pm \sqrt{a}$. For these quantities to be solvable at the left boundary $x=0$, the parameters in condition~\eqref{I_3} must satisfy the (Uniform) Kreiss Condition (\textbf{UKC})
\begin{align}\label{UKC}
	B_u+\sqrt{a}B_v\neq 0.
\end{align}
Under this assumption, the incoming flow $\sqrt{a}u^\varepsilon+v^\varepsilon$ at $x=0$ can be directly determined from the outgoing flow $(\sqrt{a}u^\varepsilon-v^\varepsilon)$ at $x=0$ and the data $b(t)$. The condition~\eqref{UKC} is a well-known necessary and sufficient criterion for the IBVP \eqref{Problem} to be well-posed for any fixed~$\varepsilon$ \cite{Benzoni-GavageSerre07}. Let us note that this criterion admits natural generalizations when handling with higher dimensions hyperbolic systems set over the multidimensional half-space $x\in\RR_+\times\RR^{d-1}$. In the present work, we focus on the simplest one-dimensional $2\times 2$ system \eqref{Problem}.\\
For consistency purposes, one can assume that the initial data $f(x)=\left(u_0(x), v_0(x)\right)$ and the boundary data $b(t)$ are compatible at the space-time corner $(x,t)=(0,0)$, for example in the following sense
\begin{align}
	f(0)=f'(0)=0,\quad
	b(0)=b'(0)=0.
\end{align}
However, this assumption is here not mandatory since we are only concerned with stability features.

\medskip

As previously discussed, when the parameter $\varepsilon$ goes to zero in~\eqref{I_1}, this system faces the \textbf{relaxation limit}. The stability of the relaxation process towards the equilibrium system in the absence of space boundary is known to be available under the Whitham subcharacteristic  condition \cite{Liu87}. For~\eqref{I_1}, this condition simply corresponds to the inequality $a>0$. In the presence of a boundary condition as~\eqref{I_3}, the situation is more tricky and the rigorous derivation of asymptotic behaviour for $u^\varepsilon$ and $v^\varepsilon$ is the key challenge from now more than two decades. Yong \cite{Yong99} addressed this problem for general multi-dimensional linear constant coefficient relaxation systems, or one-dimensional nonlinear systems, first with non-characteristic boundaries. He derived the so-called Generalized Kreiss Condition (GKC), which enables uniform stability estimates and the derivation of a reduced boundary condition for the limiting relaxed equilibrium system. 
For the particular boundary value Jin-Xin system \eqref{Problem}, with stiff source terms having the slightly more general form $\varepsilon^{-1}\left(\lambda u^\varepsilon -v^\varepsilon\right)$ where $\lambda\in\mathbb{R}$, the equilibrium system consists in the advection equation $\partial_t u + \lambda \partial_x u =0$. Xin and Xu \cite{XinXu00} identify and rigorously justify a necessary and sufficient condition on the boundary parameters $(B_u,B_v)$ that guarantees the uniform well-posedness of the corresponding IBVP, independently of the relaxation parameter $\varepsilon$. This condition is called the Stiff Kreiss Condition (SKC) and reads as a uniform lower bound satisfied by a parameterized determinant function. Hopefully, through the normal mode analysis and then a conformal mapping theorem, the abstract form of the SKC can be simplified to an explicit algebraic condition in terms of the coefficients $(B_u,B_v)$. Namely the \textbf{SKC} (or GKC) simply reduces, for the problem~\eqref{Problem}, to the following explicit condition:
\begin{align*}\label{SKC1}
	B_v=0\quad\text{or}\quad\dfrac{B_u}{B_v}\notin\left[-\sqrt{a},0\right].
\end{align*}
All along the present paper, without loss of generality by multiplying the boundary equation~\eqref{I_3} by -1, we assume that $B_u > 0$. 
The above condition then simplifies to
\begin{equation}\label{SKC}
	B_u+\sqrt{a}B_v>0,
\end{equation}
which can now be viewed as a subset of the UKC~\eqref{UKC}.
In addition to the work of Yong \cite{Yong99}, the study in~\cite{XinXu00} not only obtains the above condition but also addresses the asymptotic expansions for the limiting unknowns $u^\varepsilon$ and $v^\varepsilon$. These expansions involve both boundary and/or initial layers in the appropriate scaling of the time and space variables. For the IBVP~\eqref{Problem}, the boundary becomes characteristic in the limit and thus enters the framework of characteristic boundaries of type II in the denomination of the more recent works of Zhou and Yong~\cite{YongZhou21,ZhouYong21,ZhouYong22}. In these works, a three-scale expansion involving the space variable $(x,x/\sqrt{\varepsilon},x/\varepsilon)$ is used to fully describe the asymptotic boundary-layer behaviors in general multidimensional linear hyperbolic relaxation systems. The scale $x/\sqrt{\varepsilon}$ is precisely required when the boundary is characteristic for the equilibrium system. In that case, there is no reduced boundary condition for the limiting equilibrium system. In terms of possible applications, we refer the reader for example to \cite{CaoYong22} for the construction of boundary conditions for Jin-Xin models, to \cite{ZhaoHuangYong19} and \cite{ZhaoYong21} for boundary conditions for kinetic-based models, and to \cite{CoquelJinLiuEtAl14} and  \cite{HuangLiZhou23} for the use of relaxation models with discontinuous relaxation rate in coupling strategies. 
\medskip

The \textbf{motivation of the present study} is to analyze the counterpart of the stiff stability condition~\eqref{SKC}, but now for difference approximations of the IBVP~\eqref{Problem}. 
Naturally, any numerical approximation comes with its own stability features with respect to the time and space steps, but our interest again focuses more on the stiff stability with respect to the parameter $\varepsilon$. The reason for that is the effectiveness of approximation schemes is directly tied to the design of suitable discrete boundary conditions that ensure  stability estimate that are robust to cross along the convergence analysis, independently of $\varepsilon$. The next related step is to select, within the set of uniformly stable discrete boundary conditions, those that can minimize the size and impact of the possible artificial discrete boundary layers, while leading to (high order) accurate results. The next step should also be based not only on the choice for the discrete structure at the boundary but also on the choice of appropriate high order approximations on the discrete boundary data.

Let us now present the work done concerning this first stability feature in the previous paper \cite{BoutinNguyenSeguin20}. The (semi-)discrete boundary condition for the same model problem was implemented through the discrete version of~\eqref{I_3} supplemented with an other scalar evolution equation. This requirement of an artificial “incoming" quantity comes from the enlarged spatial stencil, both being absent in the PDE model. As a consequence, new discrete instabilities may emerge \cite{Trefethen84} in the computations. 
By using the summation-by-parts method from \cite{KreissScherer74, Strand94}, the homogeneous problem ($b=0$) is proved to satisfy natural energy estimates. These are estimates in terms of the initial data.
On the one side, these estimates are uniform in the parameters $(\Delta x,\varepsilon)$ for the case where $B_v>0$ (actually this case reads $B_uB_v>0$ in~\cite{BoutinNguyenSeguin20} since we did not assume $B_u>0$). This corresponds also to the case considered in 
\cite{LiuYong01}.
On the other side, in the case $B_v<0$, a restriction on the parameter $\Delta x/\varepsilon$ is necessary to guarantee the uniformity of the available estimate. This strict dissipativity condition can be reformulated as 
\begin{align}\label{DSDC}
	2aB_v+\tfrac{\Delta x}{\varepsilon}B_u>0.
\end{align}
In particular, the above condition precludes the possibility of having $\Delta x$ going to zero for a fixed $\varepsilon>0$. The second part of the previous work is concerned with the non-homogenous case $b\neq 0$. Using the Laplace transform and the normal mode analysis, the proposed semi-discrete approximation for \eqref{Problem} is then proved to be stiffly stable for $B_v>0$, meaning with full estimates in terms of the initial data $f$ and the boundary data $b$. The case $B_v>0$ is only a proper subset of the SKC~\eqref{SKC} and additional numerical evidences strongly support the conjecture that the scheme proposed in~\cite{BoutinNguyenSeguin20} is not stiffly stable under the only condition~\eqref{DSDC}, i.e if \eqref{SKC} is fulfilled but with $B_v<0$, even if the energy estimate is available.\\

In the present article, we construct a new family of stiffly stable finite difference schemes based on the central scheme with either the semi-discrete framework, or with the implicit discrete in time solver. The boundary treatment is again based on the SBP technique, now together with the SAT technique for imposing the physical boundary conditions in a weak sense. This technique is a penalty like one that incorporates the boundary condition as a kind of relaxation term in a boundary evolution equation. It was proposed in \cite{CarpenterGottliebAbarbanel94, CarpenterNordstromGottlieb99} and we will see that the method is strongly compatible with the obtaining of appropriate energy estimates. As a general tool, the SBP-SAT is thought to be more tractable and extendable to further extensions (e.g. high-order schemes). We now introduce the precise discrete framework in which we operate, the assumptions and the two main results.

\subsection{Description of the semi-discrete numerical scheme}
We focus in this paper on the semi-discrete approximation of the IBVP \eqref{Problem} obtained by the central differencing scheme and we derive a sufficient condition for its stiff stability. Let $\Delta x>0$ be the space step and introduce the grid points $x_j=j\Delta x,$ for any $j\in\mathbb{N}$. 
At each grid point $x_j$, the approximation of the exact solution to \eqref{Problem} is denoted by
$U_j(t) \simeq \left(u(x_j,t), v(x_j,t)\right)^T$ (where we omit the explicit dependence on $\varepsilon$). To reduce the notations, let us introduce the matrices
\begin{align*}
	A=\begin{pmatrix}
		0 & 1
		\\
		a & 0
	\end{pmatrix},\quad
	S=\begin{pmatrix}
		0 & 0
		\\
		0 & -1
	\end{pmatrix},\quad
	B=\begin{pmatrix}
		B_u & B_v
	\end{pmatrix}.
\end{align*}
The proposed semi-discrete approximation of the IBVP \eqref{Problem} is the following:
\begin{align}\label{NS}
	\begin{cases}
		\partial_t U_0(t)+ \left(\mathcal{Q}U\right)_0(t)=\varepsilon^{-1} SU_0(t) +\tfrac{2}{\Delta x}\Phi\left(BU_0(t)-b(t)\right), & t\geq 0,
		\\
		\partial_t U_j(t)+ \left(\mathcal{Q}U\right)_j(t)=\varepsilon^{-1}SU_j(t), & j\geq 1,\ t\geq 0,
		\\
		U_j(0)=f_j, & j\geq 0,
	\end{cases}
\end{align}
for a given discrete Cauchy data $f_j=U_j(0)$, $j\in\mathbb{N}$. The constant parameter vector $\Phi=(\alpha, \beta)^T$ enables a particular treatment close to the boundary which will be made explicit in the forthcoming Definition~\ref{Definition4}. Its choice is our main issue.

The finite difference operator $\mathcal{Q}$ is defined by means of the SBP technique proposed by Strand in \cite{Strand94} (see also beginnings and extensions of the idea in \cite{KreissScherer74} and \cite{GustafssonKreissOliger13}). The term $\mathcal{Q}U$ is a consistent approximation of the first order space-derivative $A\partial_x U$ in the sense that $(\mathcal{Q}U)(x_j,t)=A\partial_x U(x_j,t)+\mathcal{O}\left(\Delta x^p\right)$ for some $p>0$ ($p=2$ for the second order central approximation). The first component of the difference approximation $(\mathcal{Q}U)_0$, corresponding to the discrete boundary point $j=0$, has a somehow slightly different treatment. This adjustment in the scheme may be interpreted as the application of the same central approximation at the boundary point $j=0$, but for another boundary condition determined through a ghost value $U_{-1}$. The corresponding value is obtained from the identity $U_1-2U_0+U_{-1}=0$.
Eliminating $U_{-1}$, then we obtain a one-sided approximation for $(\mathcal{Q}U)_0$. Finally, the considered difference operator is
\begin{align}\label{operator}
	\begin{split}
		(\mathcal{Q}U)_j=\begin{cases}
			\frac{1}{2\Delta x}A\left(U_{j+1}-U_{j-1}\right), & j\geq 1,
			\\
			\frac{1}{\Delta x}A\left(U_1-U_0\right), & j=0.
		\end{cases}
	\end{split}
\end{align}

Together with the SAT technique developed by Carpenter and collaborators \cite{CarpenterGottliebAbarbanel94,CarpenterNordstromGottlieb99} for imposing the boundary condition, the SBP operator described in \eqref{operator} can be used to discretize any IBVP of the form \eqref{Problem}.
Associated to this SBP technique, an energy estimate is obtained by using the modified scalar product and norm
\begin{align}\label{scalar_product}
	\left\langle U,V\right\rangle_{\Delta x}=\dfrac{\Delta x}{2} \left\langle U_0,V_0\right\rangle
	+\Delta x\sum_{j=1}^{\infty}\left\langle U_j,V_j\right\rangle,\quad
	\Vert U\Vert^2_{\Delta x}=\left\langle U,U\right\rangle_{\Delta x}
\end{align}
with $\left\langle .,.\right\rangle $ being the usual Euclidean inner product on $\mathbb{R}^2$. We refer again to \cite{GustafssonKreissOliger13} for more details on the general technique.


\subsection{Main result}
As mentioned earlier, for the continuous IBVP~\eqref{Problem}, the usual UKC~\eqref{UKC} is insufficient to provide uniform a priori estimates in the relaxation limit. In fact, the more stringent condition SKC~\eqref{SKC} must be considered to obtain such uniform a priori estimates. Our goal is to similarly determine a sufficient condition for the stiff stability of the semi-discrete IBVP~\eqref{NS}, specifically in terms of stability estimates that remains uniform with respect to the stiffness of the relaxation term. 

\medskip

The design of appropriate discrete boundary conditions requires careful attention to the choice of the SAT-parameter vector $\Phi=(\alpha,\beta)^T$ in~\eqref{NS}. The suitable set of parameters is defined in the following definition.

\begin{df}[SAT-parameter]\label{Definition4}
	Let $B_u>0$ and $B_v\in\mathbb{R}$. The pair $(\alpha,\beta)\in\RR^2$ is called a SAT-parameter if it satisfies the following inequalities:
	\begin{align}\label{SAT1}
		\begin{cases}
			\alpha<(3+2\sqrt{2})\min(B_v^{-1},0), &\text{if }B_v\neq 0
			\\
			\alpha<0, &\text{if }B_v=0,
		\end{cases}
	\end{align}
	and
	\begin{align}\label{SAT2}
		\begin{cases}
			-a(1-B_v\alpha)-2a\sqrt{\vert B_v\alpha\vert}< \beta B_u <
			-a(1-B_v\alpha)+2a\sqrt{\vert B_v\alpha\vert}, &\text{if }B_v\neq 0
			\\
			\beta B_u=-a, &\text{if }B_v=0.
		\end{cases}
	\end{align}
\end{df}

\begin{thm}[Semi-discrete scheme]\label{main_results}
	Let $B_u>0$, $B_v\in\RR$ and $(\alpha,\beta)$ be a SAT-parameter in the sense of~Definition~\ref{Definition4}. 
	Assume that the parameters $\Delta x\in (0,1]$ and $\varepsilon >0$ satisfy the discrete strict dissipativity condition~\eqref{DSDC}. 
	%
	For any $T>0$, there exists $C_T>0$ such that for any $\left(f_j\right)_{j\in\mathbb{N}}\in \ell^2\left(\mathbb{N},\mathbb{R}^2\right)$ and any $b\in L^2\left(\mathbb{R}^+,\mathbb{R}\right)$, the solution $\left(U_j\right)_{j\in\mathbb{N}}\in\mathcal{C}^1(\RR^+,\ell^2\left(\mathbb{N},\mathbb{R}^2\right))$ to \eqref{NS} satisfies
	\begin{align}\label{main_result}
		\int_0^T\sum_{j\geq 0} \Delta x\left\vert U_j(t)\right\vert^2 \mathsf{d}t 
		+\int_0^T \left\vert U_0(t)\right\vert^2 \mathsf{d}t
		\leq C_T(\Delta x, \varepsilon)\Bigl(
			\sum_{j\geq 0}\Delta x\left\vert f_j\right\vert^2
			+\int_0^T |b(t)|^2 \,\mathsf{d}t
		\Bigr),
	\end{align}
	where the constant $C_T(\Delta x, \varepsilon)$ is independent of the data $f$ and $b$ and satisfies the uniform behaviour hereafter.
	\begin{itemize}
		\item[a)] For $B_v>0$, the inequality~\eqref{main_result} holds uniformly, i.e. with $C_T(\Delta x, \varepsilon)=C_T$ independent of $\varepsilon$ and $\Delta x$.
		\item[b)] For $B_v\leq 0$, the inequality~\eqref{main_result} holds uniformly for $\varepsilon=O(\Delta x)$, i.e. with $C_T(\Delta x, \varepsilon)=C_T(\delta_0)$, as soon as $\Delta x\geq \delta_0\varepsilon$, where $\delta_0>-4aB_u^{-1}B_v$.
	\end{itemize}
\end{thm}
In the case of full time discretization, ensuring the stability of the algorithm requires that boundary conditions be specified in accordance with the chosen time discretization method (e.g., forward Euler, backward Euler, Runge-Kutta,...). This is connected to the energy conservation of the numerical scheme, as it depends on both the structure of the problem and the discretization approach used. 

For example, we consider the simplest fully discrete approximation of the IBVP \eqref{Problem}, which is obtained by the implicit scheme in time treatment of the semi-discrete scheme~\eqref{NS}:
\begin{align}\label{NS_full}
	\begin{cases}
		\frac{1}{\Delta t}\left(U_0^{n+1}-U_0^n\right)+ \left(\mathcal{Q}U\right)^{n+1}_0=\varepsilon^{-1} SU^{n+1}_0 +\tfrac{2}{\Delta x}\Phi\left(BU^{n+1}_0-b^{n+1}\right), & n\geq 0,
		\\
		\frac{1}{\Delta t}\left(U_j^{n+1}-U_j^n\right)+ \left(\mathcal{Q}U\right)^{n+1}_j=\varepsilon^{-1}SU^{n+1}_j, & j\geq 1,\ n\geq 0,
		\\
		U^0_j=f_j, & j\geq 0,
	\end{cases}
\end{align}
with $U_j^n$ be the approximation of the exact solution to \eqref{Problem} at the grid point $(x_j, t^n)=(j\Delta x, n\Delta t)$, for any $(j,n)\in\mathbb{N}\times\mathbb{N}$.
In this case, we can prove that the discrete strict dissipativity condition \eqref{DSDC} is still the 
sufficient condition for the stiff stability of the fully discrete IBVP~\eqref{NS_full}. It is the result of the following theorem
\begin{thm}[Implicit scheme]\label{main_results_full}
	Let $B_u>0$, $B_v\in\RR$ and $(\alpha,\beta)$ be a SAT-parameter in the sense of~Definition~\ref{Definition4}. 
	Assume that the parameters $\Delta x\in (0,1]$ and $\varepsilon >0$ satisfy the discrete strict dissipativity condition~\eqref{DSDC}.
	For any $T>0$, there exists $C_T>0$ such that for all $\Delta t>0$, any $\left(f_j\right)_{j\in\mathbb{N}}\in \ell^2\left(\mathbb{N},\mathbb{R}^2\right)$ and $\left(b^n\right)_{n\in\mathbb{N}}\in \ell^2\left(\mathbb{N},\mathbb{R}\right)$, the solution $\left(U_j^n\right)_{(j,n)\in\mathbb{N}\times \mathbb{N}}$ to \eqref{NS_full} satisfies
	\begin{align}\label{main_result_full}
		\sum_{n=0}^N\sum_{j\geq 0}\Delta x\Delta t\vert U_j^n\vert^2
		+\sum_{n=1}^N\Delta t \vert U_0^n\vert^2
		\leq C_T(\Delta x, \varepsilon)\left(
		\sum_{j\geq 0} \Delta x\vert f_j\vert^2 
		+\sum_{n=1}^N\Delta t\vert b^n\vert^2
		\right)
	\end{align}
	where $N:=T/\Delta t$ and the constant $C_T(\Delta x,\varepsilon)$ is independent of the data $f$ and $b$ and satisfies the uniform behaviour hereafter.
	\begin{itemize}
		\item[a)] For $B_v>0$, the inequality~\eqref{main_result_full} holds uniformly, i.e. with $C_T(\Delta x, \varepsilon)=C_T$ independent of $\varepsilon$ and $\Delta x$.
		\item[b)] For $B_v\leq 0$, the inequality~\eqref{main_result_full} holds uniformly for $\varepsilon=O(\Delta x)$, i.e. with $C_T(\Delta x, \varepsilon)=C_T(\delta_0)$, as soon as $\Delta x\geq \delta_0\varepsilon$, where $\delta_0>-4aB_u^{-1}B_v$.
	\end{itemize}
\end{thm}
To address the stiff well-posedness of the Jin-Xin relaxation model \cite{JinXin95}, Xin and Xu derived the SKC~\eqref{SKC} in \cite{XinXu00}. Specifically, they demonstrate that the IBVP~\eqref{Problem} is well-posed if and only if the SKC~\eqref{SKC} holds.
However, in the discrete IBVP~\eqref{NS}, it appears that even the SKC is insufficient to obtain uniform stability estimates. It is worth noting that the discrete strict dissipativity condition \eqref{DSDC} is not implied by the SKC~\eqref{SKC}, likely due to some numerical diffusion at the boundary.

\bigskip

The paper is organized as follows. The Theorem~\ref{main_results} and its fully discrete counterpart Theorem~\ref{main_results_full} are studied in Section~\ref{ENGMT} using the discrete energy method. The appropriate selection of the SAT-parameter $(\alpha, \beta)$ is discussed in detail in Section~\ref{sec:SAT}, with several technical points deferred to the appendix in Section~\ref{Sec:Technical}. 
To demonstrate the importance of the condition~\eqref{DSDC}, we present numerical results in Section~\ref{Section_NE}, exploring various values of the boundary parameters $(B_u,B_v)$. These results illustrate the behaviour of solutions in both the relaxation and characteristic variables, highlighting the efficiency of the numerical method.

\section{The energy estimates for the linear damped wave equation}\label{ENGMT}

This is very classical to get the required energy estimates in the continuous case using integration by parts. Accordingly, we apply similar SBP (Summation-by-Parts) rules for the discrete approximation of $\partial /\partial x$.
Additionally, using the SAT strategy, with the choice of the SAT-parameter $(\alpha,\beta)$ as defined in Definition~\ref{Definition4}, along with certain technical lemmas in Section~\ref{Sec:Technical}, we proceed to prove our main result, Theorem~\ref{main_results}.

Throughout the following proof of Theorem~\ref{main_results}, we omit, for simplicity, the explicit dependence of the functions $U_0, u_0, v_0, b$ on the time variable $t$ and on $\varepsilon$. It should also be noted that we utilize the forthcoming technical Proposition~\ref{Lemma5} which plays a crucial role in the proof.
%

\subsection{Proof of Theorem~\ref{main_results}}

	Firstly we introduce a symmetrizer for the continuous PDE system~\eqref{I_1} that is appropriate for both the transport and the source term: this is a symmetric positive definite matrix $H$ with the properties that $HA$ is symmetric and $HS$ is negative semi-definite. The following simple matrix is convenient
	\begin{align*}
		H=\begin{pmatrix}
			a & 0
			\\
			0 & 1
		\end{pmatrix},
		\textrm{ with } 
		HA=\begin{pmatrix}
			0 & a \\ a & 0
		\end{pmatrix} \textrm{ and }
	HS=\begin{pmatrix}
		0 & 0\\ 0 & -1
	\end{pmatrix}.
	\end{align*}
	On the semi-discrete side, considering the space-discrete scalar product~\eqref{scalar_product}, and since $H$ is symmetric, we compute the evolution of the energy $E(t):=\langle U(t),HU(t)\rangle_{\Delta x}$ as follows 
	\begin{align}\label{Eng9}
		\tfrac12\partial_t E = 
		\tfrac12\partial_t\langle U,HU\rangle_{\Delta x}
		= \langle \partial_t U,HU\rangle_{\Delta x}
		=\tfrac12 \Delta x\langle \partial_t U_0,HU_0\rangle
		+\Delta x\sum_{j\geq 1}\langle \partial_t U_j,HU_j\rangle.
	\end{align}
	Now, due to the specific discrete operator $\mathcal{Q}$ in~\eqref{operator}, any solution to the semi-discrete IBVP~\eqref{NS} satisfies the following energy balance:
\begin{align}\label{Eng10}
	\begin{split}
		\tfrac12\partial_t E = 
		&=\tfrac{\Delta x}{2\varepsilon}\langle SU_0,HU_0\rangle
		+\tfrac{\Delta x}{\varepsilon}\sum_{j\geq 1}\langle SU_j,HU_j\rangle
		+ (BU_0-b)\langle \Phi , HU_0\rangle
		+\tfrac{1}{2}\langle AU_0,HU_0\rangle
		\\
		&-\tfrac{1}{2}\langle AU_1,HU_0\rangle 
		-\tfrac{1}{2}\sum_{j\geq 1}\langle AU_{j+1},HU_j\rangle
		+\tfrac{1}{2}\sum_{j\geq 1}\langle AU_{j-1},HU_j\rangle.
	\end{split}
\end{align}
Since $H$ and $HA$ are symmetric, the very last term on the right-hand side also reads
\begin{align}\label{Eng11}
		\sum_{j\geq 1}\langle AU_{j-1}, HU_j\rangle
		= \langle AU_0,HU_1\rangle +\sum_{j\geq 1}\langle AU_j, HU_{j+1}\rangle
		= \langle AU_1,HU_0\rangle +\sum_{j\geq 1}\langle AU_{j+1}, HU_{j}\rangle.
\end{align}
Substituting \eqref{Eng11} into \eqref{Eng10} removes the three last terms in \eqref{Eng10}: this is the interesting point with the SBP method.\\
The energy balance now directly comes from the remains terms in~\eqref{Eng11}, namely those involving only boundary values, except for the dissipative source term:
\begin{align*}
	\tfrac12\partial_t E 
	=
	-\tfrac{\Delta x}{\varepsilon}\sum_{j\geq 1}v_j^2
	-\tfrac{\Delta x}{2\varepsilon}v_0^2
	+\left(B_uu_0+B_vv_0-b\right)\left(\alpha au_0+\beta v_0\right)+au_0v_0.
\end{align*}
From the dissipativity of the interior relaxation term on the right-hand side, and keeping only boundary terms, one has
\begin{align}\label{Eng12s}
	\partial_t E
	\leq  -\tfrac{\Delta x}{\varepsilon}v_0^2
	+2\left(B_uu_0+B_vv_0-b\right)\left(\alpha au_0+\beta v_0\right)+2au_0v_0.
\end{align}
In order for the energy method to work, the boundary condition has now to satisfy
\begin{align}\label{Eng13}
	-\tfrac{\Delta x}{\varepsilon}v_0^2+2\left(B_uu_0+B_vv_0-b\right)\left(\alpha au_0+\beta v_0\right)+2au_0v_0\leq -C\vert U_0\vert^2+D  b^2
\end{align}
for some constants $C>0, D>0$ independent of the data and solution. Proposition~\ref{Lemma5} precisely consists in the analysis of this property, and from there we now that there exists $c>0$ such that 
\begin{align*}
	\mathcal{F}(U_0)\geq c\, \mathcal{I}(U_0)
\end{align*}
with 
\begin{equation}\label{quadratic_F}
	\mathcal{F}(U) := \tfrac{\Delta x}{\varepsilon}v^2 - 2\left(B_uu+B_vv\right)\left(\alpha au+\beta v\right) - 2auv,\qquad 
	\mathcal{I}(U) := |U|^2 = u^2+v^2.
\end{equation}
As a consequence, we obtain the following inequality
\begin{align}\label{proof_Lemma5_11}
	-\tfrac{\Delta x}{\varepsilon}v_0^2+2(B_uu_0+B_vv_0-b)(\alpha au_0+\beta v_0)+2au_0v_0
	\leq 
	-c\vert U_0\vert^2-2b(\alpha au_0+\beta v_0).
\end{align}
On the other hand, by using simple quadratic inequalities, 
we have
\begin{align}\label{proof_Lemma5_12}
	-\tfrac12c\vert U_0\vert^2-2b(\alpha au_0+\beta v_0)
	\leq 2c^{-1}(\alpha^2 a^2+\beta^2) b^2.
\end{align}
Assembling the two inequalities~\eqref{proof_Lemma5_11} and~\eqref{proof_Lemma5_12}, we get
\begin{align}\label{proof_Lemma5_13s}
	\begin{split}
		-\tfrac{\Delta x}{\varepsilon}v_0^2+2(B_uu_0+B_vv_0-b)(\alpha au_0+\beta v_0)+2au_0v_0
		\leq 
		-\tfrac{c}{2}\vert U_0\vert^2 + \tfrac{D}{c} b^2,
	\end{split}
\end{align}
where $D=2(\alpha^2 a^2+\beta^2)$.
Therefore, the energy balance~\eqref{Eng12s} implies the following inequality
\begin{align}\label{Eng14}
	\partial_t E \leq -\tfrac{c}{2}\vert U_0\vert^2 + \tfrac{D}{c} b^2,
\end{align}
and, integrating over $t\in[0,T]$, we then have:
\begin{align}\label{Eng15}
	E(T)+\tfrac{c}{2}\int_0^T\vert U_0(t)\vert^2 \mathsf{d}t 
	\leq E(0) + \tfrac{D}{c}\int_0^Tb^2(t) \mathsf{d}t.
\end{align}
Let $\gamma \geq 0$ be given and consider $\partial_t(E(t)e^{-2\gamma t})= (\partial_t E(t) - 2\gamma E(t))e^{-2\gamma t}$ together with~\eqref{Eng14} to get, after integrating over $t\in[0,T]$ the new weighted estimate:
\begin{equation}\label{fullestimate}
	e^{-2\gamma T}E(T) + 2\gamma \int_0^T E(t)e^{-2\gamma t}\mathsf{d}t + \tfrac{c}{2} \int_0^T |U_0(t)|^2 e^{-2\gamma t}\mathsf{d}t \leq E(0) + \tfrac{D}{c} \int_0^T b^2(t) e^{-2\gamma t}\mathsf{d}t.
\end{equation}
For $\gamma=0$, we obviously recover the previous estimate~\eqref{Eng15}, but for $\gamma>0$ we obtain after crude bounds on the exponential growth terms:
\begin{equation}\label{eng16}
	E(T) + 2\gamma \int_0^T E(t)\mathsf{d}t + \tfrac{c}{2} \int_0^T |U_0(t)|^2\mathsf{d}t \leq e^{2\gamma T}\left(E(0) + \tfrac{D}{c} \int_0^T b^2(t)\mathsf{d}t\right).
\end{equation}
Finally, choosing a fixed value for $\gamma$, since $H$ is symmetric positive definite and from the previous inequality,
for any $T>0$, there exists a constant $C_T>0$ (depending on $T$ and on $H$, $\gamma$, $c$ and $D$) such that the following inequality holds
\begin{align}
	\int_0^T\sum_{j\geq 0}\Delta x\vert U_j(t)\vert^2 \mathsf{d}t
	+\int_0^T\vert U_0(t)\vert^2 \mathsf{d}t
	\leq C_T\left(
	\sum_{j\geq 0}\Delta x\vert f_j\vert^2
	+\int_0^T b^2(t) \mathsf{d}t
	\right).
\end{align}
This ends the proof of Theorem \ref{main_results}.

\subsection{Proof of Theorem~\ref{main_results_full}}	
We use similar techniques as in \eqref{Eng9}-\eqref{Eng11} to cover the time-implicit discretization. Any solution to the fully-discrete IBVP~\eqref{NS_full} satisfies the following energy balance
\begin{align}\label{Eng12p1}
	\begin{split}
	\tfrac{1}{\Delta t}\langle U^{n+1}-U^n,HU^{n+1}\rangle_{\Delta x}
\leq  &-\tfrac{\Delta x}{2\varepsilon}\left(v_0^{n+1}\right)^2
+\left(B_uu^{n+1}_0+B_vv^{n+1}_0-b^{n+1}\right)\left(\alpha au^{n+1}_0+\beta v^{n+1}_0\right)
\\
&+au^{n+1}_0v^{n+1}_0.
\end{split}
\end{align}
Moreover, since $H$ is symmetric positive definite matrix, the time-dissipation estimate for the implicit Euler method reads as follows:
\begin{align}\label{Eng12p}
	\begin{split}
	\langle U^{n+1}-U^n,HU^{n+1}\rangle_{\Delta x}
	&= \tfrac{1}{2}\left(
	\langle U^{n+1},HU^{n+1}\rangle_{\Delta x}
	-\langle U^{n},HU^{n}\rangle_{\Delta x}
	+\langle U^{n+1}-U^n,H\left(U^{n+1}-U^n\right)\rangle_{\Delta x}
		\right)
	\\
	&\geq \tfrac{1}{2}\left(
	\langle U^{n+1},HU^{n+1}\rangle_{\Delta x}
	-\langle U^{n},HU^{n}\rangle_{\Delta x}\right).
\end{split}
\end{align}
According to \eqref{Eng12p1} and \eqref{Eng12p}, we obtain the following inequality where we set $E_n := \langle U^{n},HU^{n}\rangle_{\Delta x}$:
\begin{align}\label{Eng12}
	\begin{split}
		\tfrac{1}{\Delta t}\left(E_{n+1}
		-E_n\right)
		\leq  &-\tfrac{\Delta x}{\varepsilon}\left(v_0^{n+1}\right)^2
		+2\left(B_uu^{n+1}_0+B_vv^{n+1}_0-b^{n+1}\right)\left(\alpha au^{n+1}_0+\beta v^{n+1}_0\right)
		+2au^{n+1}_0v^{n+1}_0.
	\end{split}
\end{align}
Let us mention that the right-hand side is now nothing but the discrete version at time $t^{n+1}$ of the right-hand side of inequality \eqref{Eng12s}. Therefore, no particular change in the analysis is required. Using the property~\eqref{proof_Lemma5_13s}, we easily have the discrete energy balance
\begin{align*}
\tfrac{1}{\Delta t}\left(E_{n+1}
-E_n\right) \leq -\tfrac{c}{2}\vert U_0^{n+1}\vert^2 + \tfrac{D}{c} \left(b^{n+1}\right)^2.
\end{align*}
As a consequence, the previous inequality becomes
\begin{align}\label{Eng15d}
	E_n
	+\tfrac{c}{2}\Delta t\sum_{k=1}^{n}\vert U_0^k\vert^2
	\leq E_0 
	+\tfrac{D}{c}\Delta t\sum_{k=1}^{n}\vert b^k\vert^2,\quad
	\text{for any } n>0.
\end{align}
Let us now fix some $T>0$, and consider integer $N$ such that $N\Delta t \leq T$. Since $H$ is symmetric positive definite and since the SBP-norm~\eqref{scalar_product} is uniformly equivalent to the usual $\Delta x$-weighted $\ell^2$-norm over the space $\ell^2(\mathbb{N},\RR^2)$, we infer from~\eqref{Eng15d} the two following inequalities :
\[\begin{aligned}
	\tfrac{c}{2}\Delta t\sum_{k=1}^{N}\vert U_0^k\vert^2
	& \leq C_0 \sum_{j\geq 0}\Delta x |f_j|^2 
	+\tfrac{D}{c}\Delta t\sum_{k=1}^{N}\vert b^k\vert^2,\\
	C_0^{-1} \sum_{k=0}^{N} \sum_{j\geq 0}\Delta t\Delta x |U_j^n|^2 &\leq (N+1)\Delta t \left(C_0\sum_{j\geq 0}\Delta x |f_j|^2 + \tfrac{D}{c}\Delta t\sum_{k=1}^{N}\vert b^k\vert^2 \right).
\end{aligned}\]
Assembling the two inequalities above, we obtain the estimate~\eqref{main_result_full} with $C_T>0$ (depending linearly on $T$ and also on $C_0, c, C_0$ and $D$). This ends the proof of Theorem \ref{main_results_full}.
\section{Choice of the SAT parameters}
\label{sec:SAT}
The simultaneous approximation term (SAT) technique, introduced in \cite{CarpenterGottliebAbarbanel94, CarpenterNordstromGottlieb99}, enforces boundary conditions weakly through a penalty-like term, which also facilitates the time-stability of the approximation. In this work, we employ the SAT technique by selecting the scalar parameters $\alpha$ and $\beta$ in~\eqref{NS}, which will be discussed in detail later. More precisely, we now state the primary property resulting from the choice of the SAT-parameter in Definition~\ref{Definition4}, particularly regarding the useful inequality \eqref{Eng13}. Several technical aspects in the upcoming proof are deferred to the appendix, Section~\ref{Sec:Technical}, including Lemma~\ref{Lemma2_appendix}, Lemma~\ref{lemma6_appendix} and Lemma~\ref{lemma7_appendix}.
\begin{pro}\label{Lemma5}
	Let $(\alpha,\beta)$ be SAT-parameter in the sense of Definition~\ref{Definition4}. Assume that the parameters $\Delta x\in (0,1]$, $\varepsilon>0$, $B_u>0$ and $B_v\in\RR$ satisfy the discrete strict dissipativity condition~\eqref{DSDC}.
	Then the quadratic form $\mathcal{F}$ in~\eqref{quadratic_F} is positive definite, meaning there exists a constant $c(\Delta x, \varepsilon)>0$ such that
	\begin{equation}
		\label{result_lemma5}
		\mathcal{F} \geq c(\Delta x, \varepsilon)\, \mathcal{I}. 
	\end{equation}
	More precisely, 
	\begin{itemize}
		\item[a)] For $B_v>0$, $\mathcal{F}$ is uniformly positive definite, i.e. with $c(\Delta x, \varepsilon)=c$ independent of $\varepsilon$ and $\Delta x$.
		\item[b)] For $B_v\leq 0$,$\mathcal{F}$ is positive definite uniformly for $\varepsilon=O(\Delta x)$, i.e. with $c(\Delta x, \varepsilon)=c(\delta_0)$, as soon as $\Delta x\geq\delta_0\varepsilon$, where $\delta_0>-4aB_u^{-1}B_v$.
	\end{itemize}
\end{pro}
\begin{proof}
	In the following proof, we simply denote $c(\Delta x, \varepsilon)=c$ the desired constant. 
	From the definition of $\mathcal{F}$ in~\eqref{quadratic_F}, the inequality~\eqref{result_lemma5} in fact corresponds to the following property
	\begin{equation}
		HA + \tfrac{\Delta x}{\varepsilon} HS + 2 \Re (H\Phi B) \leq -c I_2,
	\end{equation}
	or in the explicit coordinates:
	\begin{align}\label{proof_Lemma5_1}
		\forall (u,v)\in\RR^2\quad \left(
		2B_v\beta + c -\tfrac{\Delta x}{\varepsilon}
		\right)v^2+2\left(
		B_u\beta +B_v\alpha a+a
		\right)u v
		+\left(2B_u\alpha a+c\right) u^2\leq 0.
	\end{align}
	
	\medskip

	First, we start with some straightforward considerations based on examining the diagonal terms only. By considering the specific case where $v = 0$, the required inequality~\eqref{proof_Lemma5_1} reduces to $2B_u\alpha a + c \leq 0$. This condition is satisfied for some $c>0$ because, from the condition on $\alpha$ in ~\eqref{SAT1}, we have $\alpha<0$.\\
	Now, let us consider the specific case with $u = 0$. In this situation, the required inequality~\eqref{proof_Lemma5_1} becomes $2B_v\beta + c -\frac{\Delta x}{\varepsilon}\leq 0$.	According to Lemma~\ref{Lemma1_appendix}, under the choice of the SAT-parameter $(\alpha,\beta)$ from Definition~\ref{Definition4} and the condition~\eqref{DSDC}, we have $B_v\beta<0$ for $B_v\neq 0$. Thus, there exists $c\in (0,-2B_v\beta)$ that satisfies this inequality, independently of the values of $\Delta x$ and $\varepsilon$. For $B_v = 0$, the requirement $\Delta x = O(\varepsilon)$ clearly emerges as necessary to achieve a similar result. This implies that if there exists $\delta_0>0$ such that $\Delta x/\varepsilon\geq \delta_0$, we can choose $c\in (0,\delta_0)$.

	\medskip

	Now, we consider the general case with $u\neq 0$ and set $X=v/u$. The inequality \eqref{proof_Lemma5_1} then becomes
	\begin{align*}
		\left(
		2B_v\beta + c -\tfrac{\Delta x}{\varepsilon}
		\right)X^2+2\left(
		B_u\beta +B_v\alpha a+a
		\right)X
		+2B_u\alpha a+ c \leq 0.
	\end{align*}
	It holds true for any $X\in\mathbb{R}$ if and only if the following inequalities are satisfied 
	\begin{align*}
		\begin{cases}
			2B_v\beta + c -\tfrac{\Delta x}{\varepsilon}< 0
			\\
			\left(
			B_u\beta +B_v\alpha a+a
			\right)^2-	\left(
			2B_v\beta +c -\tfrac{\Delta x}{\varepsilon}
			\right)\left(
			2B_u\alpha a+c 
			\right)\leq 0.
		\end{cases}
	\end{align*}
	It is equivalent to the following system
	\begin{align}\label{proof_Lemma5_3}
		\begin{split}
			\begin{cases}
				c< -2B_v\beta +\tfrac{\Delta x}{\varepsilon}
				\\
				c^2+\left(
				2B_v\beta +2B_u\alpha a -\tfrac{\Delta x}{\varepsilon}
				\right)c
				+2B_u\alpha a\left(
				2B_v\beta -\tfrac{\Delta x}{\varepsilon}
				\right)
				-\left(
				B_u\beta+B_v\alpha a+a
				\right)^2\geq 0.
			\end{cases}
		\end{split}
	\end{align}
	We now consider the above algebraic expression as a quadratic polynomial in $c$. Let us define
	\begin{align}\label{Definition_Delta}
		\begin{split}
			\Delta:&= \left(
			2B_v\beta +2B_u\alpha a-\tfrac{\Delta x}{\varepsilon}
			\right)^2-4\left[
			2B_u\alpha a\left(
			2B_v\beta -\tfrac{\Delta x}{\varepsilon}
			\right)
			-\left(
			B_u\beta+B_v\alpha a+a
			\right)^2
			\right]
			\\
			&= \left(
			2B_v\beta -2B_u\alpha a-\tfrac{\Delta x}{\varepsilon}
			\right)^2
			+4\left(
			B_u\beta+B_v\alpha a+a
			\right)^2.
		\end{split}
	\end{align}
	Considering the roots of the above polynomial, the system \eqref{proof_Lemma5_3} is equivalent to
	\begin{align}\label{bonus}
		\begin{cases}
			c< -2B_v\beta +\tfrac{\Delta x}{\varepsilon}
			\\
			\left(
			c+B_v\beta+B_u\alpha a-\tfrac{\Delta x}{2\varepsilon}+\tfrac{1}{2}\sqrt{\Delta}
			\right)
			\left(
			c+ B_v\beta+B_u\alpha a-\tfrac{\Delta x}{2\varepsilon} -\tfrac12\sqrt{\Delta}
			\right)\geq 0.
		\end{cases}
	\end{align}
	From the definition of $\Delta$ in \eqref{Definition_Delta} , we have for all $\alpha,\beta\in\mathbb{R}$
	\begin{align}\label{proof_Lemma5_7}
		-B_v\beta-B_u\alpha a+\tfrac{\Delta x}{2\varepsilon}-\tfrac12\sqrt{\Delta}
		\leq
		-2B_v\beta +\tfrac{\Delta x}{\varepsilon}
		\leq
		-B_v\beta-B_u\alpha a+\tfrac{\Delta x}{2\varepsilon}+\tfrac12\sqrt{\Delta}.
	\end{align}
	On the other hand, from the result of Lemma~\ref{Lemma2_appendix}, we have
	\begin{align}\label{proof_Lemma5_6}
		-B_v\beta-B_u\alpha a+\tfrac{\Delta x}{2\varepsilon}-\tfrac12\sqrt{\Delta}>0.
	\end{align}
	According to \eqref{proof_Lemma5_7} and \eqref{proof_Lemma5_6}, the system~\eqref{bonus} is true provided that $c$ satisfies
	\begin{align}\label{proof_Lemma5_8}
		0<c\leq -B_v\beta-B_u\alpha a+\tfrac{\Delta x}{2\varepsilon}-\tfrac12\sqrt{\Delta}.
	\end{align}
	Now, we look at the following three situations occur according to the sign of $B_v$.
	\begin{itemize}
		\item If $B_v>0$ then, by  Lemma~\ref{lemma6_appendix}, one has
		\begin{align}\label{proof_Lemma5_9}
			\begin{split}
				0&<-B_v\beta - B_u\alpha a  -\sqrt{(B_v\beta -B_u\alpha a)^2+(B_u\beta +B_v\alpha a+a)^2}
				\\
				&<-B_v\beta-B_u\alpha a+\tfrac{\Delta x}{2\varepsilon}-\sqrt{\left(B_v\beta-B_u\alpha a-\tfrac{\Delta x}{2\varepsilon}\right)^2+(B_u\beta +B_v\alpha a+a)^2}.		
			\end{split}
		\end{align}
		According to \eqref{proof_Lemma5_8} and \eqref{proof_Lemma5_9}, the inequality \eqref{result_lemma5} holds with  
		\begin{align*}
			c\in \left(0, -B_v\beta -B_u\alpha a -\sqrt{(B_v\beta -B_u\alpha a)^2+(B_u\beta +B_v\alpha a+a)^2}\right].
		\end{align*}
		\item If $B_v<0$ then, by Lemma~\ref{lemma7_appendix}, one has
		\begin{align}\label{proof_Lemma5_10}
			\begin{split}
				0&<-B_v\beta -B_u\alpha a +\tfrac{\delta_0}{2}-\sqrt{\bigl(B_v\beta -B_u\alpha a-\tfrac{\delta_0}{2}\bigr)^2+(B_u\beta +B_v\alpha a+a)^2}
				\\
				&<
				-B_v\beta-B_u\alpha a+\tfrac{\Delta x}{2\varepsilon}-\sqrt{\left(B_v\beta-\tfrac{\Delta x}{2\varepsilon}-B_u\alpha a\right)^2+(B_u\beta +B_v\alpha a+a)^2}
			\end{split}
		\end{align}
		with $\Delta x\geq \delta_0\varepsilon$ and $\delta_0>-4aB_u^{-1}B_v$.
		\\
		According to \eqref{proof_Lemma5_8} and \eqref{proof_Lemma5_10} the inequality \eqref{result_lemma5} holds with
		\begin{align*}
			c\in \left(0,-B_v\beta -B_u\alpha a +\tfrac{\delta_0}{2}-\sqrt{\bigl(B_v\beta -B_u\alpha a-\tfrac{\delta_0}{2}\bigr)^2+(B_u\beta +B_v\alpha a+a)^2}\right].
		\end{align*}
		\item If $B_v=0$ then, under the condition on $\beta$ in \eqref{SAT2}, the inequality \eqref{proof_Lemma5_8} can be reformulated as
		\begin{align}\label{proof_Lemma5_10_2}
			0<c\leq -B_u \alpha a+\tfrac{\Delta x}{2\varepsilon}-\sqrt{\left(B_u\alpha a+\tfrac{\Delta x}{2\varepsilon}\right)^2}
		\end{align}
		On the other hand, under the condition on $\alpha$ in \eqref{SAT1}, one has
		\begin{align}\label{proof_Lemma5_10_1}
			0<-B_u \alpha a+\tfrac{\delta_0}{2}-\sqrt{(B_u\alpha a+\tfrac{\delta_0}{2})^2}
			<-B_u \alpha a+\tfrac{\Delta x}{2\varepsilon}-\sqrt{(B_u\alpha a+\tfrac{\Delta x}{2\varepsilon})^2}
		\end{align}
		with $\Delta x\geq \delta_0\varepsilon$ and $\delta_0>0$.
		\\
		According to \eqref{proof_Lemma5_10_2} and \eqref{proof_Lemma5_10_1}, the inequality \eqref{result_lemma5} holds with $c$ such that
		\begin{align*}
			c\in\left(0, -B_u \alpha a+\tfrac{\delta_0}{2}-\sqrt{(B_u\alpha a+\tfrac{\delta_0}{2})^2}\right].
		\end{align*}
	\end{itemize}
	This ends the proof of Proposition \ref{Lemma5}.
\end{proof}

\begin{remark}
	The detailed proof above and the useful Lemmas in the appendix convincingly suggests that the restricted choice of the SAT-parameters in Definition~\ref{Definition4} is, in a sense, optimal (i.e. maximal) for ensuring the discrete strict dissipativity condition holds. However, there is no unique way to include the boundary condition through a discrete SAT-strategy as done here in the scheme~\eqref{NS}.
\end{remark}
%

\section{Numerical experiment}\label{Section_NE}
In this section, we present several numerical experiments. All experiments are conducted with the parameters $a=4$ over the space interval $x\in[0,2]$. Each initial data function $f$ we consider (nearly) vanishes outside the space interval $[0,0.7]$. Consequently, due to the characteristic velocities $\pm \sqrt{a}=\pm 2$, no non-trivial information reaches the right computational boundary until time $T=0.6$. For this reason, we restrict the computation of numerical solutions to the time domain $t\in[0,0.6]$. At the right point $x=2$, we impose a discrete boundary condition based on the first-order homogeneous Neumann boundary condition, namely $U_{J+1}(t)=U_{J}(t)$ where $J$ denotes the rightmost cell.

To avoid in-depth analysis of time integration for stiff ODEs and to stay within the scope of this work, we use a build-in time solver and handle the semi-discrete scheme~\eqref{NS}. The chosen time solver is LSODA, wich uses Adams/BDF method with automatic stiffness detection and switching, from the ODEPACK detailed in \cite{Hindmarsh83,Petzold83}. The fully discrete implicit scheme~\eqref{NS_full} is not used since the numerical computation on a bounded interval will inescapably couple the whole space domain at each time step and thus miss the stiff stability analysis done on the half-line \cite{InglardLagoutiereRugh20}.

Finally, the SAT-parameter $\Phi=(\alpha,\beta)^T$ in~\eqref{NS} satisfies the conditions~\eqref{SAT1}–\eqref{SAT2} of Definition~\ref{Definition4}. In most of the illustrative numerical experiments, the boundary parameters satisfy the strict dissipativity condition~\eqref{DSDC} with $B_u>0$ and $B_v\neq 0$. A very few tests will be conducted outside that condition. For the sake of reproducibility in the experiments, we mention that the values for the SAT-parameter $(\alpha,\beta)$ are systematically chosen so as to fulfil the requirements in Definition~\ref{Definition4}, as follows:
\[
	\alpha = (3+2\sqrt{2})\min(B_v^{-1},0) - 2,\qquad \beta = -a(1-B_v\alpha)B_u^{-1}.
\]

\subsection{Energy behaviour}

We first demonstrate the efficiency of the discrete strict dissipativity condition \eqref{DSDC}. Specifically, we observe the time evolution of the discrete energy $E(t)=\langle U(t),HU(t)\rangle_{\Delta x}$ which depends on the values of $\varepsilon$ and the boundary parameter $(B_u,B_v)$. The following initial and boundary data are used:
\begin{equation}
	f(x) = \begin{pmatrix}15\\10\end{pmatrix} \chi_{(0,1/2]}(x),\qquad b(t)=0,
\end{equation}
along with the space step $\Delta x= 5.10^{-3}$, which corresponds to $J = 2/\Delta x = 400$ cells.
We consider two regimes $\varepsilon=10^{-2}$ and $\varepsilon=10^2$, and for each, a corresponding set of values $(B_u,B_v)$ is chosen according to the discrete strict dissipativity condition \eqref{DSDC}. The only where the condition \eqref{DSDC} is not satisfied is when $\varepsilon=10^2$ and $(B_u,B_v)=(20,-1)$. The evolution of the energy $E(t)$ over the time period $t \in [0,0.6]$ is shown in Figure~\ref{Figure1}.
\begin{figure}[ht]
	\begin{center}
		\includegraphics[scale=.38,clip,trim=30 20 50 50]{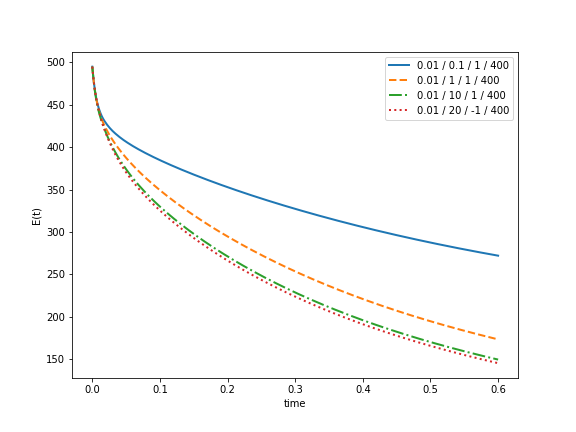}
		\includegraphics[scale=.38,clip,trim=30 20 50 50]{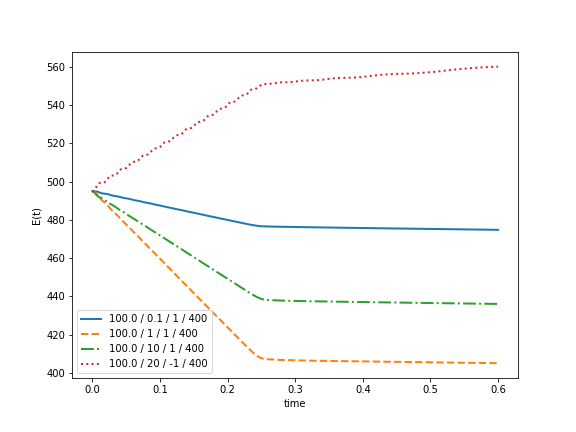}
		\caption{Evolution of the energy $E(t)$ for $\varepsilon=10^{-2}$ (left) and $\varepsilon=10^2$ (right). The legends are the parameters $\varepsilon / B_u / B_v / J$.}
		\label{Figure1}
	\end{center}
\end{figure}
Let us now comment on these results.
\begin{itemize}
	\item In the proof of Theorem~\ref{main_results}, we observe that the energy $E(t)$ decreases for a vanishing boundary condition $b=0$, provided the parameters satisfy the discrete strict dissipativity condition~\eqref{DSDC}. The numerical experiments strongly support this observation, including the case when $B_uB_v<0$, such as $(B_u,B_v)=(20,-1)$ and $\varepsilon=10^{-2}$, but only when $\Delta x$ is sufficiently large relative to $\varepsilon$. The case $B_uB_v<0$ but without the strict dissipativity condition~\eqref{DSDC} is simulated using $(B_u,B_v)=(20,-1)$ and a large $\varepsilon=10^2$. In this scenario, the energy clearly increases at first, despite the vanishing data at the left boundary.
	\item For small values of $\varepsilon$, the energy $E(t)$ decreases exponentially fast over short times due to the initial relaxation towards equilibrium. In the case of the larger value $\varepsilon=10^2$, the energy decay follows a linear behaviour, with the transition time $t=1/4$ corresponding to the time at which the non-trivial transported quantities exit the left boundary (part of which then returns to the interior domain). When dealing with the pathological case that does not satisfy~\eqref{DSDC}, the numerical solution exhibits a strong discrete reflected wave as long as physical quantities reach the left boundary. As a result, the energy increases within the time interval $[0,1/4]$.
\end{itemize}

\subsection{Convergence experiments}

We examine the behaviour of the solution for several values of the parameters $\varepsilon\in\{10^{-3},5.10^{-2},10^2\}$ and the space step $\Delta x\in\{0.04,0.02,0.01,0.005\}$. Throughout the simulations, the boundary condition for the exact IBVP is set to $(B_u,B_v)=(1,1)$. Therefore, the strict dissipativity condition~\eqref{DSDC} is always satisfied, with no restrictions on $\Delta x$ or $\varepsilon$. The physical initial data and boundary data are set to
\begin{equation}
	f(x) = \begin{pmatrix}2 e^{-100(0.5-x)^2}\\4 e^{-100(0.4-x)^2}\end{pmatrix},\qquad b(t)=5\sin^2(4\pi t).
\end{equation}
Figure~\ref{Figure2a1} illustrates the solutions in the space-time plane for the values of the relaxation parameter $\varepsilon\in\{10^{2},5.10^{-2},10^{-3}\}$ and with $\Delta x= 0.005$. The color field in the figure represents the magnitude of the unknown $u$, $v$, as well as the characteristic variables $\sqrt{a}u+v$ and $\sqrt{a}u-v$. Additionally, Figure~\ref{Figure2b} shows several quantities of interest: the unknowns $u$ and $v$ at the final time $t=0.6$, the energy $E(t)$ as a function of time, and the boundary quantities in the boundary cell $BU_0(t)-b(t)$ and in the first interior cell $BU_1(t)-b(t)$. Let us now discuss these results further.

\begin{itemize}
	\item For the largest value of $\varepsilon=10^2$ (Figure~\ref{Figure2a1}), the characteristic variables $\sqrt{a}u\pm v$ clearly evolve at the boundary according to the underlying hyperbolic structure, with outgoing and incoming waves, respectively. Meanwhile, the unknowns $u$ and $v$ follow the boundary condition~\eqref{I_3}. As the relaxation parameter $\varepsilon$ decreases, the influence of the hyperbolic structure diminishes, with $v$ vanishing and approaching the equilibrium value of zero, except in the small spatial layer near  the boundary $x=0$ (as seen in Figure~\ref{Figure2b}, particularly in the three plots on the second line).
	\item On each subplot of Figure~\ref{Figure2b}, several curves are plotted, corresponding to progressively smaller space steps $\Delta x$. The results show both the boundedness and the convergence of the plotted quantities. Specifically, the boundary condition $BU_0(t)-b(t)$ converges to zero (4th row), regardless of the value of $\varepsilon$. However, for small values of $\varepsilon$, the interaction between the relaxation layer and the boundary layer slows the convergence of the first cell quantity $BU_1(t)-b(t)$ (5th row), and in the relaxation limit, this convergence appears to vanish. The key issus here is related to the difficulty in establishing trace existence for the limiting characteristic problem.
\end{itemize}

\begin{figure}[!ht]
	\begin{center}
		\includegraphics[scale=.40,clip,trim=50 360 360 72]{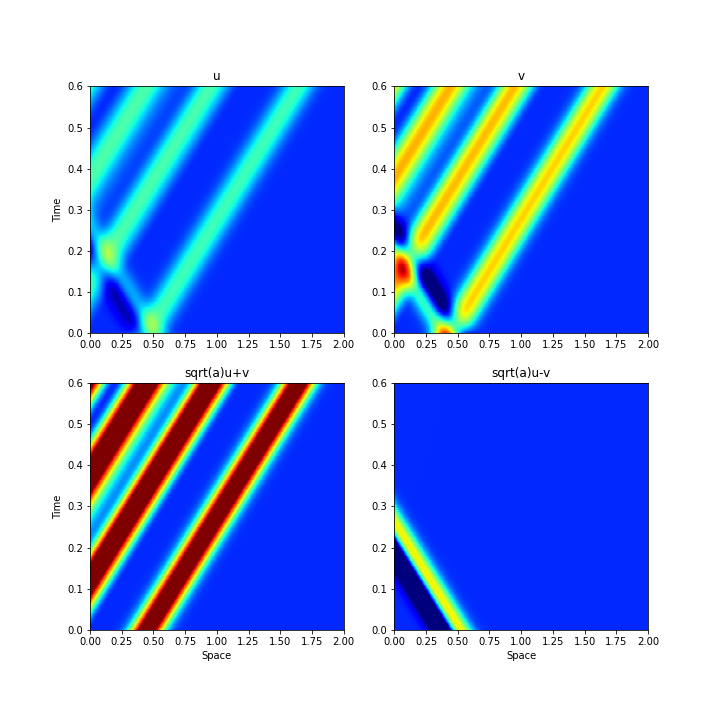}
		\includegraphics[scale=.40,clip,trim=50 360 360 72]{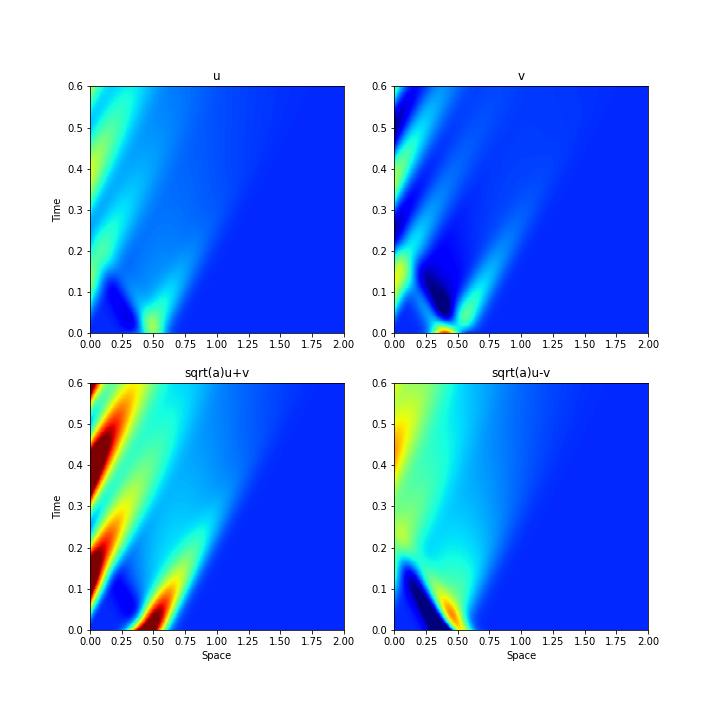}
		\includegraphics[scale=.40,clip,trim=50 360 360 72]{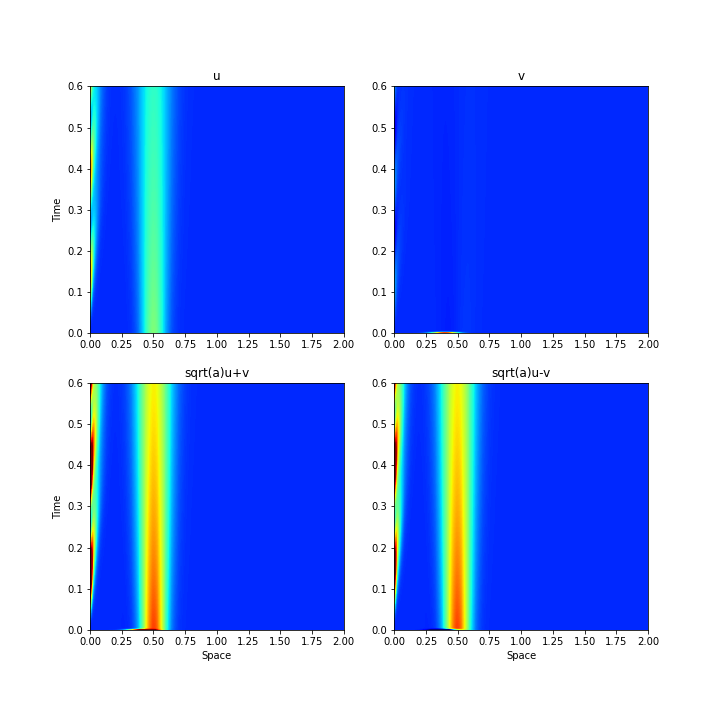}\\
		\includegraphics[scale=.40,clip,trim=355 360 55 72]{100.png}
		\includegraphics[scale=.40,clip,trim=355 360 55 72]{005.png}
		\includegraphics[scale=.40,clip,trim=355 360 55 72]{0001.png}\\
		\includegraphics[scale=.40,clip,trim=50 60 360 365]{100.png}
		\includegraphics[scale=.40,clip,trim=50 60 360 365]{005.png}
		\includegraphics[scale=.40,clip,trim=50 60 360 365]{0001.png}\\
		\includegraphics[scale=.40,clip,trim=355 60 55 365]{100.png}
		\includegraphics[scale=.40,clip,trim=355 60 55 365]{005.png}
		\includegraphics[scale=.40,clip,trim=355 60 55 365]{0001.png}
		\caption{Space-time behaviour of $u$, $v$, $\sqrt{a}u+v$ and $\sqrt{a}u-v$ (from top to bottom). Relaxation parameter: $\varepsilon=10^{2}$, $\varepsilon=5.10^{-2}$, $\varepsilon = 10^{-3}$ (from left to right). Fixed parameters: $J=800$. $B=(1,1)^T$.}
		\label{Figure2a1}
	\end{center}
\end{figure}

\begin{figure}[!hp]
	\begin{center}
		\includegraphics[scale=.25,clip,trim=20 20 50 50]{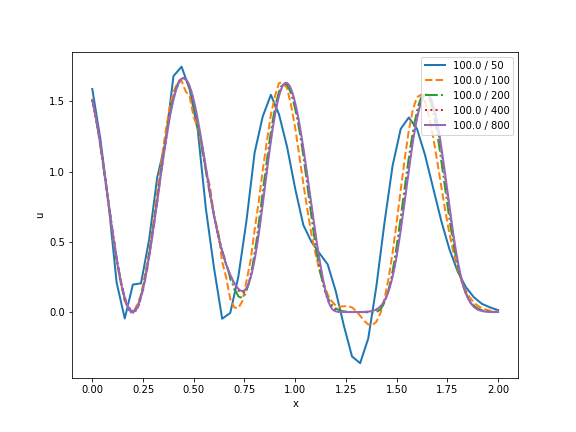}
		\includegraphics[scale=.25,clip,trim=20 20 50 50]{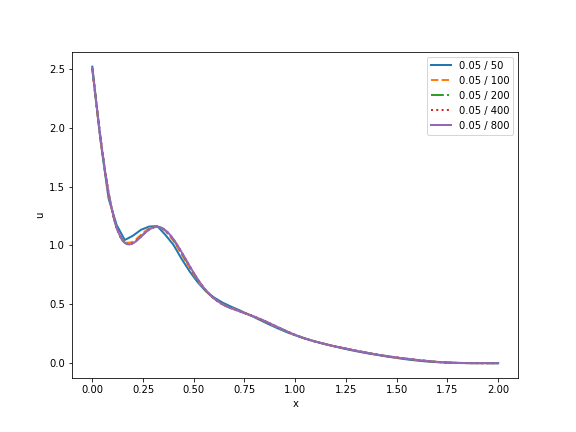}
		\includegraphics[scale=.25,clip,trim=20 20 50 50]{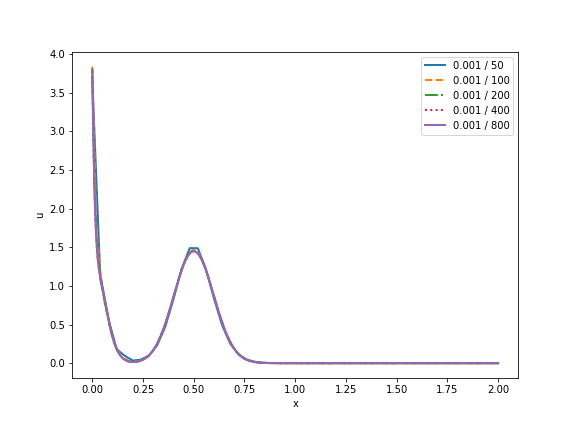}\par
		\includegraphics[scale=.25,clip,trim=20 20 50 50]{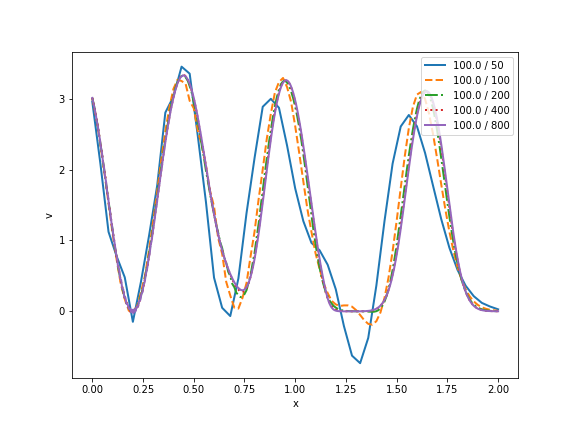}
		\includegraphics[scale=.25,clip,trim=20 20 50 50]{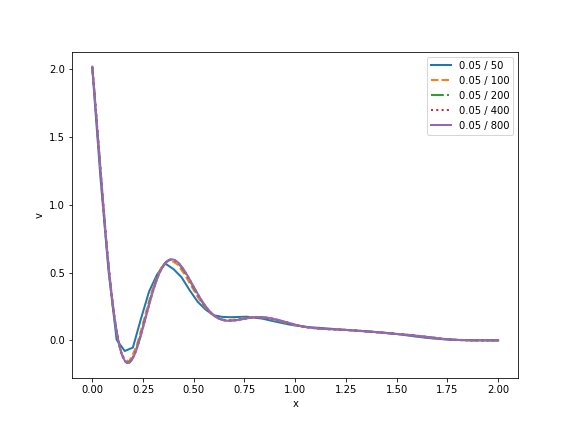}
		\includegraphics[scale=.25,clip,trim=20 20 50 50]{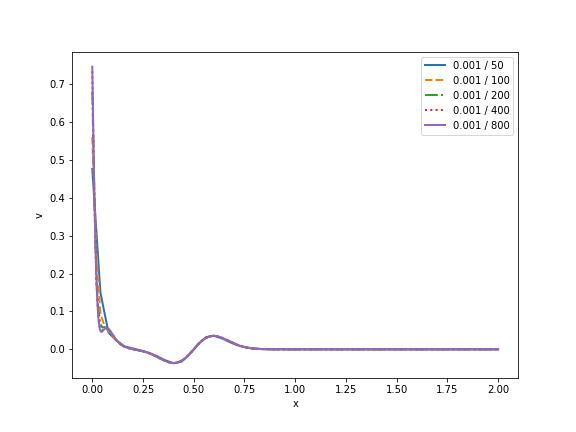}\par
		\includegraphics[scale=.25,clip,trim=20 20 50 50]{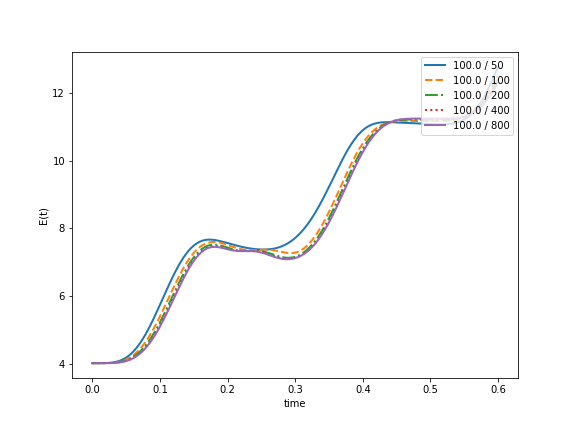}%
		\includegraphics[scale=.25,clip,trim=20 20 50 50]{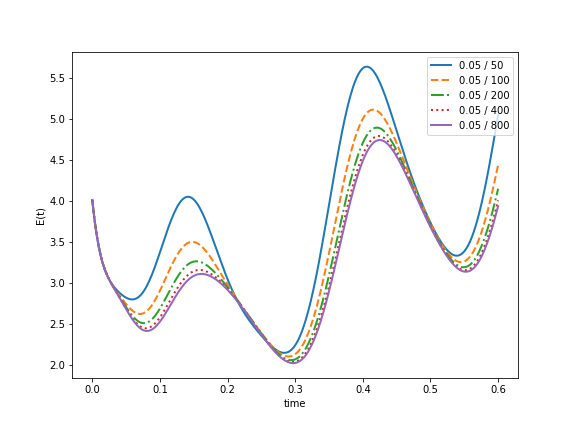}
		\includegraphics[scale=.25,clip,trim=20 20 50 50]{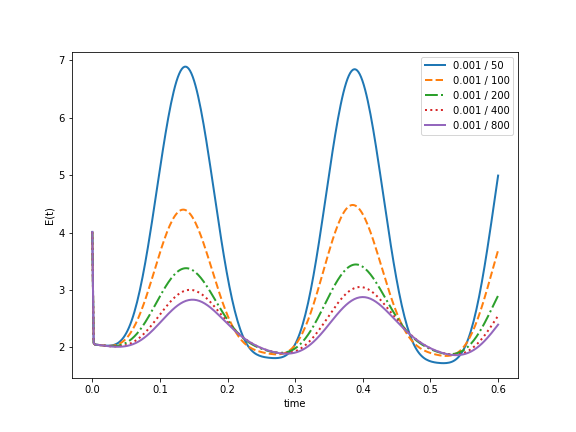}\par
		\includegraphics[scale=.25,clip,trim=20 20 50 50]{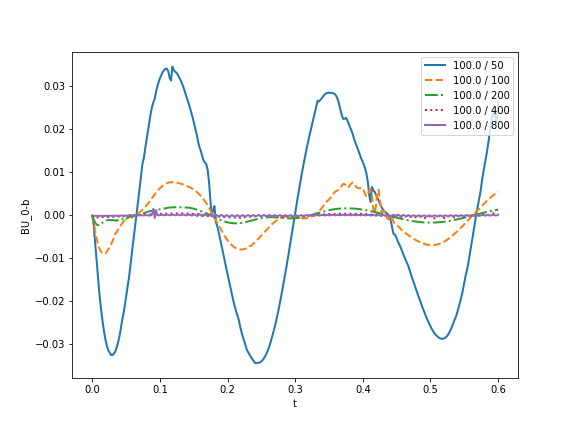}%
		\includegraphics[scale=.25,clip,trim=20 20 50 50]{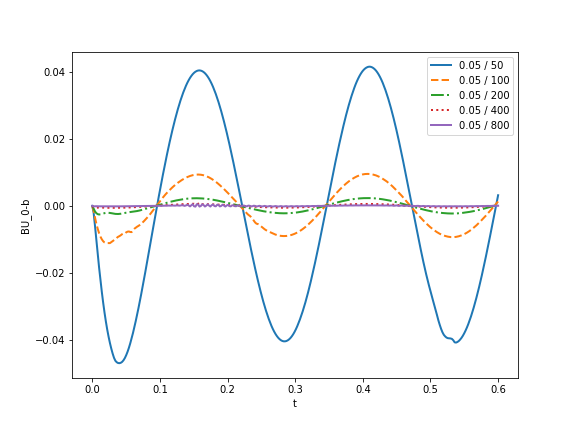}
		\includegraphics[scale=.25,clip,trim=20 20 50 50]{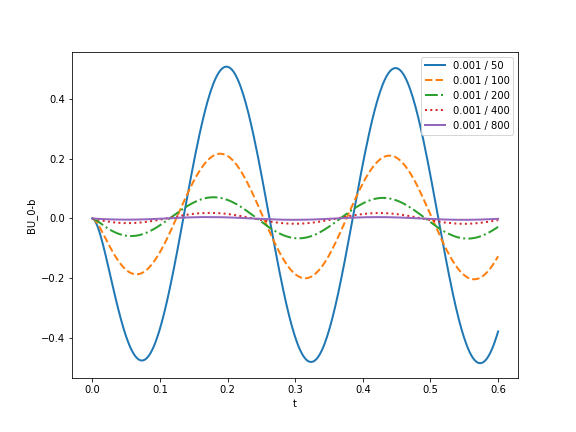}\par
		\includegraphics[scale=.25,clip,trim=20 20 50 50]{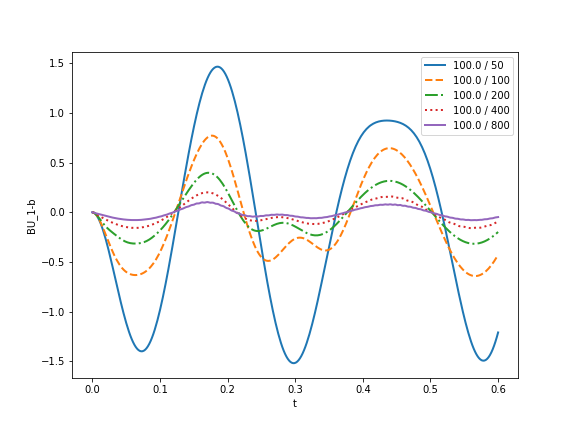}%
		\includegraphics[scale=.25,clip,trim=20 20 50 50]{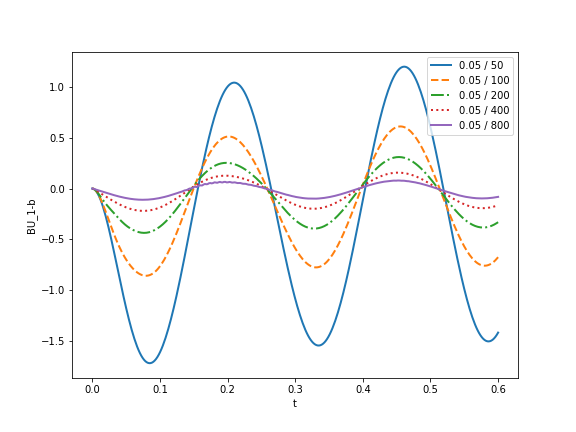}
		\includegraphics[scale=.25,clip,trim=20 20 50 50]{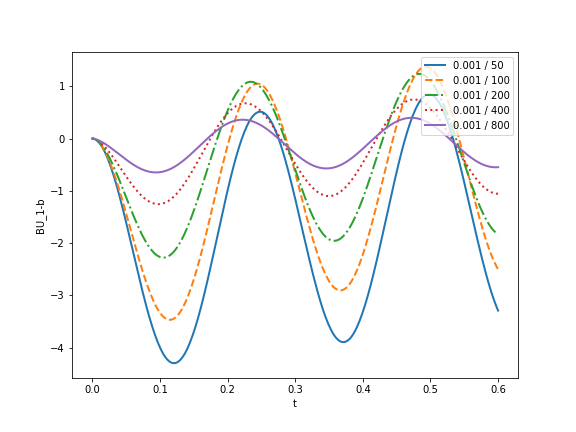}
		\caption{In the legends are the parameters $\varepsilon / J$. Fixed boundary parameter: $B=(1,1)^T$. Representation at time $t=0.6$ of $u_j$ and $v_j$, then for $t\in[0,0.6]$ of $E(t)$, $BU_0(t)-b(t)$ and $BU_1(t)-b(t)$ (from top to bottom). Relaxation parameters: $\varepsilon=10^{2}$, $\varepsilon=5.10^{-2}$, $\varepsilon = 10^{-3}$ (from left to right).}
		\label{Figure2b}
	\end{center}
\end{figure}

\clearpage
\appendix
\section{Technical lemmas}\label{Sec:Technical}
\begin{lm}\label{Lemma1_appendix}
	Let $B_u>0$, $B_v\neq 0$, $\Delta x\in (0,1]$, and $\varepsilon>0$.
	Let $(\alpha,\beta)$ be a SAT-parameter in the sense of Definition~\ref{Definition4}. 
	Then,
	\begin{align}\label{Result_Lemma1_Appendix}
		B_v\beta<0.
	\end{align}
\end{lm}
\begin{proof}
	\ \\[-1em]
	\begin{itemize}
		\item Let us first consider the case $B_v>0$. Then, the required inequality \eqref{Result_Lemma1_Appendix} is equivalent to $\beta<0$.
		To prove that, from the condition on $\beta$ in \eqref{SAT2}, we check that
		%
		$	-a(1-B_v\alpha)+2a\sqrt{\vert B_v\alpha\vert}\leq 0$,
		or equivalently that
		\begin{align}\label{Proof_Lemma1_appendix2}
			\begin{split}
				& 2\sqrt{\vert B_v\alpha\vert}\leq 1-B_v\alpha.
			\end{split}
		\end{align}
		Otherwise, according to the condition on $\alpha$ in \eqref{SAT1}, we get $\alpha<0$ 
		and then 
		$1-B_v\alpha>0$.
		Therefore, the inequality \eqref{Proof_Lemma1_appendix2} is now equivalent to
		%
			$-4 B_v\alpha \leq (1-B_v\alpha)^2$,
		%
		and then to the trivial inequality
		$(1+B_v\alpha)^2 \geq 0$.
		So
		~\eqref{Result_Lemma1_Appendix} holds. 
		\medbreak
		
		\item Secondly, we consider the case $B_v<0$. Then, the required inequality \eqref{Result_Lemma1_Appendix} is equivalent to $\beta>0$. Again to prove that, from the condition on $\beta$ in \eqref{SAT2}, we check that
		%
		%
			$-a(1-B_v\alpha)-2a\sqrt{\vert B_v\alpha\vert}\geq0$, 
		or equivalently that
		\begin{equation}\label{Proof_Lemma1_appendix11}
			2\sqrt{\vert B_v\alpha\vert}\leq-(1-B_v\alpha).
		\end{equation}
		Otherwise, according to the condition on $\alpha$ in \eqref{SAT1}, we get  $\alpha<(3+2\sqrt{2})B_v^{-1}<0$ 
		and then 
		$\alpha B_v>1$.
		Since $B_v<0$ and the previous property,
		the inequality \eqref{Proof_Lemma1_appendix11} is now equivalent to
		%
			$4B_v\alpha\leq (1-B_v\alpha)^2$, 
		or then to 
		%
			$B_v^2\alpha^2-6B_v\alpha +1\geq 0$.
		%
		It holds true if and only if
		$\alpha\in\left(
		-\infty,(3+2\sqrt{2})B_v^{-1}
		\right]\cup\left[
		(3-2\sqrt{2})B_v^{-1},+\infty
		\right).$
		From the condition on $\alpha$ in \eqref{SAT1}, $\alpha$ belongs to the first set. So, we 
		can conclude that~\eqref{Result_Lemma1_Appendix} also holds.
	\end{itemize}
	This ends the proof of Lemma \ref{Lemma1_appendix}.
\end{proof}
\begin{lm}\label{Lemma3_appendix}
	Let $B_u>0$, $B_v>0$, $\Delta x\in (0,1]$, and $\varepsilon>0$.
	Let $\alpha$ 
	be as~\eqref{SAT1} in 
	Definition~\ref{Definition4}.
	Assume that the parameters satisfy the discrete strict dissipativity condition~\eqref{DSDC}. 
	Then
	\begin{align}\label{result_Lemma3_appendix}
		-a(1-B_v\alpha)+\sqrt{-2a\left(2aB_v+\tfrac{\Delta x}{\varepsilon}B_u\right)\alpha}
		<-\tfrac{aB_u^2}{B_v}\alpha+\tfrac{B_u}{B_v}\tfrac{\Delta x}{2\varepsilon} .
	\end{align}
\end{lm}
\begin{proof}
	Since we assume $B_v>0$, the required inequality \eqref{result_Lemma3_appendix} also reads
	\begin{align*}
		B_v\sqrt{-2a\left(2aB_v+\tfrac{\Delta x}{\varepsilon}B_u\right)\alpha}
		<-aB_u^2\alpha+\tfrac{\Delta x}{2\varepsilon}B_u+aB_v(1-B_v\alpha),
	\end{align*}
	and	is equivalent to the following system
	\begin{align}\label{proof_lemma3_appendix1}
		\begin{cases}
			-aB_u^2\alpha+\tfrac{\Delta x}{2\varepsilon}B_u+aB_v(1-B_v\alpha)>0
			\\
			-2aB_v^2\left(2aB_v+\tfrac{\Delta x}{\varepsilon}B_u\right)\alpha
			<\left(
			-aB_u^2\alpha+\tfrac{\Delta x}{2\varepsilon}B_u+aB_v(1-B_v\alpha)
			\right)^2.
		\end{cases}
	\end{align}
	Firstly, we consider the second inequality in \eqref{proof_lemma3_appendix1}. It can be equivalently reformulated as
	\begin{align}\label{proof_lemma3_appendix2}
		\begin{split}
			a^2\left(B_u^2+B_v^2\right)^2\alpha^2-a\left(2aB_v+\tfrac{\Delta x}{\varepsilon}B_u\right)(B_u^2-B_v^2)\alpha +\left(\tfrac{\Delta x}{2\varepsilon}B_u+aB_v\right)^2>0.
		\end{split}
	\end{align}
	Looking at the quadratic polynomial in $\alpha$, let us define
	\begin{align*}
		\Delta_1:&=a^2\left(2aB_v+\tfrac{\Delta x}{\varepsilon}B_u\right)^2(B_u^2-B_v^2)^2
		-4a^2\left(B_u^2+B_v^2\right)^2\left(\tfrac{\Delta x}{2\varepsilon}B_u+aB_v\right)^2
		\\
		&=-4a^2B_u^2B_v^2\left(2aB_v+\tfrac{\Delta x}{\varepsilon}B_u\right)^2.
	\end{align*}
	Since $ a^2\left(B_u^2+B_v^2\right)^2>0$ and $\Delta_1<0$, the inequality \eqref{proof_lemma3_appendix2} holds for all $\alpha\in \mathbb{R}$.
	\\
	Secondly, we consider the first inequality in \eqref{proof_lemma3_appendix1}. It can be equivalently reformulated as
	\begin{align*}
		\begin{split}
			a\left(B_u^2+B_v^2\right)\alpha <
			aB_v+\tfrac{\Delta x}{2\varepsilon}B_u.
		\end{split}
	\end{align*}
	This is equivalent to
	\begin{align}\label{proof_lemma3_appendix3}
		\alpha <\dfrac{aB_v+\tfrac{\Delta x}{2\varepsilon}B_u}{a(B_u^2+B_v^2)}.
	\end{align}
	On the one side, the strict dissipativity assumption implies that the right-hand side is positive. On the other side, from the condition on $\alpha$ in \eqref{SAT1}, one has $\alpha<0$ and thus~\eqref{proof_lemma3_appendix3} holds. 
	\\
	This ends the proof of Lemma~\ref{Lemma3_appendix}.
\end{proof}

\begin{lm}\label{Lemma4_appendix}
	Let $B_u>0$, $B_v\neq 0$, $\Delta x\in (0,1]$, and $\varepsilon>0$.
	Let $(\alpha,\beta)$ be a SAT-parameter in the sense of Definition~\ref{Definition4}.
	Assume that the parameters satisfy the discrete strict dissipativity condition~\eqref{DSDC}. 
	Then,
	\begin{align}\label{result_Lemma4_Appendix1}
		\sqrt{2a\vert B_v\alpha\vert}
		<	\sqrt{-\left(2aB_v+\tfrac{\Delta x}{\varepsilon}B_u\right)\alpha}.
	\end{align}
	More precisely, 
	\begin{itemize}
		\item[a)] If $B_v>0$, the inequality \eqref{result_Lemma4_Appendix1} holds uniformly.
		\item[b)] If $B_v<0$, the inequality \eqref{result_Lemma4_Appendix1} holds uniformly for $\varepsilon=O(\Delta x)$, i.e. being given $\delta_0>-4aB_u^{-1}B_v$, as soon as $\Delta x\geq \delta_0\varepsilon$.
	\end{itemize}

\end{lm}
\begin{proof}
	The required inequalities \eqref{result_Lemma4_Appendix1} can also be rephrased as
	\begin{align}\label{proof_Lemma4_Appendix1}
		\begin{split}
			&2a \vert B_v\alpha\vert <-\left(2aB_v+\tfrac{\Delta x}{\varepsilon}B_u\right)\alpha.
		\end{split}
	\end{align}
	\begin{itemize}
		\item[a)] Consider the case $B_v>0$. From~\eqref{SAT1}, one has $\alpha<0$ and thus
		$\vert B_v\alpha\vert =-B_v\alpha$. As a consequence, the inequality \eqref{proof_Lemma4_Appendix1} equivalently reads 
		\begin{align}\label{proof_Lemma4_Appendix3}
			-2aB_v\alpha<-\left(2aB_v+\tfrac{\Delta x}{\varepsilon}B_u\right)\alpha
			\Leftrightarrow B_u\alpha <0.
		\end{align}
		Since $B_u>0$
		, the inequality \eqref{proof_Lemma4_Appendix3} holds.
		\item[b)] Consider now the case $B_v< 0$.
		Again from~\eqref{SAT1}, one has $\alpha<0$ and thus
		$\vert B_v\alpha\vert =B_v\alpha$. As a consequence, the inequality  \eqref{proof_Lemma4_Appendix1}equivalently reads 
		\begin{align}\label{proof_Lemma4_Appendix7}
			2aB_v\alpha<-\left(2aB_v+\tfrac{\Delta x}{\varepsilon}B_u\right)\alpha
			\Leftrightarrow -4aB_u^{-1}B_v\varepsilon <\Delta x.
		\end{align}
		Considering $\delta_0>-4aB_u^{-1}B_v$ such that $\Delta x\geq \delta_0\varepsilon$,
		the inequality \eqref{proof_Lemma4_Appendix7} holds.
	\end{itemize}
	This ends the proof of Lemma \ref{Lemma4_appendix}.\hspace{10cm}
\end{proof}


\begin{lm}\label{Lemma5_appendix}
	Let $B_u>0$, $B_v<0$, $\Delta x\in (0,1]$, and $\varepsilon>0$.
	Let $\alpha$ be defined as in Definition~\ref{Definition4}.
	Assume that the parameters satisfy the discrete strict dissipativity condition~\eqref{DSDC}. 
	Then,
	\begin{align}\label{result_Lemma5_appendix}
		-\tfrac{aB_u^2}{B_v}\alpha +\tfrac{B_u}{B_v}\tfrac{\Delta x}{2\varepsilon}
		<-a(1-B_v\alpha)-\sqrt{-2a\left(2aB_v+\tfrac{\Delta x}{\varepsilon}B_u\right)\alpha}.
	\end{align}
\end{lm}
\begin{proof}
	Since we assume $B_v<0$, the required inequality \eqref{result_Lemma5_appendix} also reads
	\begin{align*}
		-B_v\sqrt{-2a\left(2aB_v+\tfrac{\Delta x}{\varepsilon}B_u\right)\alpha}<-aB_u^2\alpha+\tfrac{\Delta x}{2\varepsilon}B_u+aB_v(1-B_v\alpha),
	\end{align*}
	and is equivalent to the following system
	\begin{align}\label{proof_lemma5_appendix1}
		\begin{cases}
			-aB_u^2\alpha+\tfrac{\Delta x}{2\varepsilon}B_u+aB_v(1-B_v\alpha)>0
			\\
			-2aB_v^2\left(2aB_v+\tfrac{\Delta x}{\varepsilon}B_u\right)\alpha
			<\left(
			-aB_u^2\alpha+\tfrac{\Delta x}{2\varepsilon}B_u+aB_v(1-B_v\alpha)
			\right)^2.
		\end{cases}
	\end{align}
	Firstly, we consider the second inequality in \eqref{proof_lemma5_appendix1}. It can be equivalently reformulated  as
	\begin{align}\label{proof_lemma5_appendix2}
		\begin{split}
			a^2\left(B_u^2+B_v^2\right)^2\alpha^2-a\left(2aB_v+\tfrac{\Delta x}{\varepsilon}B_u\right)(B_u^2-B_v^2)\alpha +\left(\tfrac{\Delta x}{2\varepsilon}B_u+aB_v\right)^2>0.
		\end{split}
	\end{align}
	Looking at the quadratic polynmial in $\alpha$, let us define
	\begin{align*}
		\Delta_1:&=a^2\left(2aB_v+\tfrac{\Delta x}{\varepsilon}B_u\right)^2(B_u^2-B_v^2)^2
		-4a^2\left(B_u^2+B_v^2\right)^2\left(\tfrac{\Delta x}{2\varepsilon}B_u+aB_v\right)^2
		\\
		&=-4a^2B_u^2B_v^2\left(2aB_v+\tfrac{\Delta x}{\varepsilon}B_u\right)^2.
	\end{align*}
	Since $ a^2\left(B_u^2+B_v^2\right)^2>0$ and $\Delta_1<0$, the inequality \eqref{proof_lemma5_appendix2} holds for all $\alpha\in \mathbb{R}$.
	\\
	Secondly, we consider the first inequality in \eqref{proof_lemma5_appendix1}. It can be equivalently reformulated  as
	%
			$a\left(B_u^2+B_v^2\right)\alpha < aB_v+\tfrac{\Delta x}{2\varepsilon}B_u$.
	%
	This is equivalent to
	\begin{align}\label{proof_lemma5_appendix3}
		\alpha <\dfrac{2aB_v+\tfrac{\Delta x}{\varepsilon}B_u}{2a(B_u^2+B_v^2)}.
	\end{align}
	On the one side, the discrete strict dissipativity assumption implies that the right-hand side above is positive. On the other side, from the condition on $\alpha$ in \eqref{SAT1}, one has $\alpha<0$. Thus, the inequality ~\eqref{proof_lemma5_appendix3} holds.
	\\
	This ends the proof of Lemma~\ref{Lemma5_appendix}. 
\end{proof}

\begin{lm}\label{Lemma2_appendix}
	Let $B_u>0$, $B_v\in\RR$, $\Delta x\in (0,1]$, and $\varepsilon>0$.
	Let $(\alpha,\beta)$ be a SAT-parameter in the sense of Definition~\ref{Definition4}.
	Assume that the parameters satisfy the discrete strict dissipativity condition~\eqref{DSDC}. 
	Then, 
	\begin{align}\label{result_Lemma2_appendix}
		2B_v\beta+2B_u\alpha a-\tfrac{\Delta x}{\varepsilon}+\sqrt{\Delta}<0
	\end{align}
	where $\Delta$ is defined in \eqref{Definition_Delta}.
\end{lm}
\begin{proof}
	Under the definition of $\Delta$ in \eqref{Definition_Delta}, the required inequality \eqref{result_Lemma2_appendix} is equivalent to
	\begin{align}\label{proof_Lemma2_System}
		\begin{cases}
			2B_v\beta+2B_u\alpha a-\tfrac{\Delta x}{\varepsilon}<0
			\\
			\left(2B_v\beta-2B_u\alpha a-\tfrac{\Delta x}{\varepsilon}\right)^2+4\left(
			B_u\beta +B_v\alpha a+a
			\right)^2
			<\left(
			2B_v\beta+2B_u\alpha a-\tfrac{\Delta x}{\varepsilon}
			\right)^2.
		\end{cases}
	\end{align}
	The second inequality in \eqref{proof_Lemma2_System} can be reformulated as
	\begin{align*}
			(B_u\beta)^2+2a(1-B_v\alpha)B_u\beta+a^2(1+B_v\alpha)^2+2a\tfrac{\Delta x}{\varepsilon}B_u\alpha<0.
	\end{align*}
	On the other side, under the discrete strict dissipativity condition~\eqref{DSDC}
	and the condition on $\alpha$ in~\eqref{SAT1}, we have
	\begin{align}\label{proof_Lemma2_appendix2}
		-a(1-B_v\alpha)-\sqrt{-2a\left(2aB_v+\tfrac{\Delta x}{\varepsilon}B_u\right)\alpha}
		< \beta B_u <
		-a(1-B_v\alpha)+\sqrt{-2a\left(2aB_v+\tfrac{\Delta x}{\varepsilon}B_u\right)\alpha}.
	\end{align}
	\begin{itemize}
		\item[a)] Consider the case $B_v>0$. 
		From Lemma~\ref{Lemma3_appendix}, we have
		\begin{align*}
			-a(1-B_v\alpha)+\sqrt{-2a\left(2aB_v+\tfrac{\Delta x}{\varepsilon}B_u\right)\alpha}
			<-\dfrac{aB_u^2}{B_v}\alpha+\dfrac{B_u}{B_v}\dfrac{\Delta x}{2\varepsilon} .
		\end{align*}
		On the other hand, the first inequality in~\eqref{proof_Lemma2_System} also reads
		\begin{equation*}
			B_u\beta<-\dfrac{aB_u^2}{B_v}\alpha + \dfrac{B_u}{B_v}\dfrac{\Delta x}{2\varepsilon} .
		\end{equation*}
		Hence, the system~\eqref{proof_Lemma2_System} simply consists in the bounds~\eqref{proof_Lemma2_appendix2}.\\
		Now, from Lemma \ref{Lemma4_appendix}, we deduce the two following inequalities
		\begin{align*}
			&-a(1-B_v\alpha)-\sqrt{-2a\left(2aB_v+\tfrac{\Delta x}{\varepsilon}B_u\right)\alpha}
			<-a(1-B_v\alpha)-2a\sqrt{\vert B_v\alpha\vert}\\
			&\textrm{and}\\
			&-a(1-B_v\alpha)+2a\sqrt{\vert B_v\alpha\vert}
			<	-a(1-B_v\alpha)+\sqrt{-2a\left(2aB_v+\tfrac{\Delta x}{\varepsilon}B_u\right)\alpha}.
		\end{align*}
		Therefore, the system \eqref{proof_Lemma2_System} holds with the value of $\beta$ as in ~\eqref{SAT2}.
		
		\item[b)] Consider the case $B_v<0$. 
		From Lemma~\ref{Lemma5_appendix}, we have now
		\begin{align*}
			-\dfrac{aB_u^2}{B_v}\alpha +\dfrac{B_u}{B_v}\dfrac{\Delta x}{2\varepsilon}
			<-a(1-B_v\alpha)-\sqrt{-2a\left(2aB_v+\tfrac{\Delta x}{\varepsilon}B_u\right)\alpha}.
		\end{align*}
		Again, due to $B_v<0$, the first inequality in~\eqref{proof_Lemma2_System} also reads
		\begin{equation*}
			B_u\beta>-\dfrac{aB_u^2}{B_v}\alpha+ \dfrac{B_u}{B_v}\dfrac{\Delta x}{2\varepsilon}.
		\end{equation*}
		Hence, the system~\eqref{proof_Lemma2_System} simply consists in the bounds~\eqref{proof_Lemma2_appendix2}. 
		The same analysis as the previous one is available. 
		Therefore, the system \eqref{proof_Lemma2_System} holds under the condition on $\beta$ in~\eqref{SAT2}.
		\item[c)] Consider the case $B_v=0$. The inequality ~\eqref{proof_Lemma2_appendix2} can be rewritten as
		\begin{align*}
			-a-\sqrt{-2a\tfrac{\Delta x}{\varepsilon}B_u\alpha}<\beta B_u 
			< -a+\sqrt{-2a\tfrac{\Delta x}{\varepsilon}B_u\alpha}.
		\end{align*}
		From the condition on $\beta$ in \eqref{SAT2}, one has $\beta B_u=-a$. Thus, the above inequality holds.
		\\
		On the other hand, in the case $B_v=0$, the first inequality in \eqref{proof_Lemma2_System} also reads
		\begin{align}\label{proof_Lemma2_appendix2_t}
			2B_u\alpha a-\tfrac{\Delta x}{\varepsilon}<0.
		\end{align}
	From the condition on $\alpha$ in \eqref{SAT1}, one has $\alpha <0$ and thus \eqref{proof_Lemma2_appendix2_t} holds.
	\end{itemize}

	This ends the proof of Lemma \ref{Lemma2_appendix}.
\end{proof}

\begin{lm}\label{lemma6_appendix}
	Let $B_u>0$, $B_v>0$, $\Delta x\in (0,1]$, and $\varepsilon>0$.
	Let $(\alpha,\beta)$ be a SAT-parameter in the sense of Definition~\ref{Definition4}.
	Then
	\begin{align}\label{result_lemma6_appendix2}
		\begin{split}
			&-B_v\beta-B_u\alpha a-\sqrt{(B_v\beta-B_u\alpha a)^2+(B_u\beta +B_v\alpha a+a)^2}
			\\
			<&-B_v\beta-B_u\alpha a+\tfrac{\Delta x}{2\varepsilon}-\sqrt{\left(B_v\beta-B_u\alpha a-\tfrac{\Delta x}{2\varepsilon}\right)^2+\left(B_u\beta +B_v\alpha a+a\right)^2}
		\end{split}
	\end{align}
	and
	\begin{align}\label{result_lemma6_appendix1}
		0<-B_v\beta-B_u\alpha a-\sqrt{(B_v\beta-B_u\alpha a)^2+(B_u\beta +B_v\alpha a+a)^2}.
	\end{align}
\end{lm}
\begin{proof}
	Firstly, we prove the inequality \eqref{result_lemma6_appendix2}. It is equivalent to
	\begin{align*}
		-(B_v\beta-B_u\alpha a)<\sqrt{(B_v\beta-B_u\alpha a)^2+(B_u\beta+B_v\alpha a+a)^2}.
	\end{align*}
	Since the above inequality holds for all $\alpha, \beta\in\mathbb{R}$, we can obtain the inequality \eqref{result_lemma6_appendix2}.
	\\
	Secondly, we prove that
	\begin{align}\label{proof_lemma6_appendix1}
		B_v\beta+B_u\alpha a+\sqrt{(B_v\beta-B_u\alpha a)^2+(B_u\beta +B_v\alpha a+a)^2}<0.
	\end{align}
	It is equivalent to the following system
	\begin{align}\label{proof_lemma6_appendix2}
		\begin{cases}
			B_v\beta+B_u\alpha a<0
			\\
			(B_v\beta-B_u\alpha a)^2+(B_u\beta +B_v\alpha a+a)^2<(B_v\beta+B_u\alpha a)^2
		\end{cases}
	\end{align}
	We first consider the second inequality in \eqref{proof_lemma6_appendix2}
	\begin{align*}
		(B_v\beta-B_u\alpha a)^2+(B_u\beta +B_v\alpha a+a)^2<(B_v\beta+B_u\alpha a)^2.
	\end{align*}
	It is equivalent to the following inequality
	\begin{align}\label{proof_lemma6_appendix3}
		B_u^2\beta^2+2aB_u(1-B_v\alpha)\beta +a^2(1+B_v\alpha)^2<0.
	\end{align}
	Let us define
	\begin{align*}
		\Delta_2:&=a^2B_u^2(1-B_v\alpha)^2-a^2B_u^2(1+B_v\alpha)^2
		=-4a^2B_u^2B_v\alpha.
	\end{align*}
	From the condition on $\alpha$ in~\eqref{SAT1}, one has $\alpha<0$ and therefore
	$
	\Delta_2>0.
	$

	Consequently, the inequality \eqref{proof_lemma6_appendix3} is equivalent to
	\begin{align*}
		B_u\beta\in \left(-a(1-B_v\alpha)-2a\sqrt{\vert B_v\alpha\vert},
		-a(1-B_v\alpha)+2a\sqrt{\vert B_v\alpha\vert}
		\right).
	\end{align*}
	Therefore, the second inequality in \eqref{proof_lemma6_appendix2} holds with the value of $\beta$ as in~\eqref{SAT2}.
	\\
	We now consider the first inequality in \eqref{proof_lemma6_appendix2}
	$		B_v\beta + B_u\alpha a<0.$
	It is equivalent to
	\begin{align}\label{proof_lemma6_appendix4}
		\beta<-\dfrac{aB_u}{B_v}\alpha.
	\end{align}
	We now check that the value of $\beta$ as in~\eqref{SAT2} satisfies \eqref{proof_lemma6_appendix4}.
	Since $B_u>0$ and $B_v>0$, we need to prove
	\begin{align*}
		&-a(1-B_v\alpha)+2a\sqrt{\vert B_v\alpha\vert} \leq -\dfrac{aB_u^2}{B_v}\alpha
		\\
		\Leftrightarrow\,
		& 2aB_v\sqrt{\vert B_v\alpha\vert} \leq -aB_u^2\alpha + aB_v(1-B_v\alpha).
	\end{align*}
	It is equivalent to the following system
	\begin{align}\label{proof_lemma6_appendix5}
		\begin{cases}
			-aB_u^2\alpha + aB_v(1-B_v\alpha)>0
			\\
			4a^2B_v^2\vert B_v\alpha\vert \leq \left(-aB_u^2\alpha + aB_v(1-B_v\alpha)\right)^2
		\end{cases}
	\end{align}
	Under the condition on $\alpha$ in Definition~\ref{Definition4}, we have $\alpha<0$. Hence $\vert B_v\alpha\vert =-B_v\alpha$. Therefore, the second inequality in the system \eqref{proof_lemma6_appendix5} becomes
	\begin{align}\label{proof_lemma6_appendix6}
		\begin{split}
			-4a^2B_v^2 B_v\alpha \leq \left(-aB_u^2\alpha + aB_v(1-B_v\alpha)\right)^2
			\Leftrightarrow\,
			(B_u^2+B_v^2)^2\alpha^2-2B_v(B_u^2-B_v^2)\alpha +B_v^2>0.
		\end{split}
	\end{align}
	Let us define
	\begin{align}\label{Definition_delta3}
		\begin{split}
			\Delta_3:=B_v^2(B_u^2-B_v^2)^2-B_v^2(B_u^2+B_v^2)^2
			=-4B_u^2B_v^4<0.
		\end{split}
	\end{align}
	Since $(B_u^2+B_v^2)^2>0$ and $\Delta_3<0$, the inequality \eqref{proof_lemma6_appendix6} holds for all $\alpha\in\mathbb{R}$. Then, we consider the first inequality in the system \eqref{proof_lemma6_appendix5} 
	\begin{align}\label{proof_lemma6_appendix7}
		-aB_u^2\alpha + aB_v(1-B_v\alpha)>0
		\Leftrightarrow\,\alpha <\dfrac{B_v}{B_u^2+B_v^2}.
	\end{align}
	Since $B_v>0$ and the definition of $\alpha$ in Definition~\ref{Definition4}, the inequality \eqref{proof_lemma6_appendix7} holds.
	\\
	This ends the proof of Lemma \ref{lemma6_appendix}.\hspace{10cm}
\end{proof}
\begin{lm}\label{lemma7_appendix}
	Let $B_u>0$, $B_v<0$, $\Delta x\in (0,1]$, and $\varepsilon>0$.
	Let $(\alpha,\beta)$ be a SAT-parameter in the sense of Definition~\ref{Definition4}.
	Assume that the parameters satisfy the discrete strict dissipativity condition~\eqref{DSDC}. %
	Then
	\begin{align}\label{relsult_lemma7_appendix1}
		\begin{split}
			&-B_v\beta -B_u\alpha a +\tfrac{\delta_0}{2}-\sqrt{\bigl(B_v\beta -B_u\alpha a-\tfrac{\delta_0}{2}\bigr)^2+(B_u\beta +B_v\alpha a+a)^2}
			\\
			<&
			-B_v\beta-B_u\alpha a+\tfrac{\Delta x}{2\varepsilon}-\sqrt{\left(B_v\beta-B_u\alpha a-\tfrac{\Delta x}{2\varepsilon}\right)^2+(B_u\beta +B_v\alpha a+a)^2}
		\end{split}
	\end{align}
	and
	\begin{align}\label{relsult_lemma7_appendix2}
		0<-	B_v\beta -B_u\alpha a +\tfrac{\delta_0}{2}-\sqrt{\bigl(B_v\beta -B_u\alpha a-\tfrac{\delta_0}{2}\bigr)^2+(B_u\beta +B_v\alpha a+a)^2},
	\end{align}
	with $\Delta x\geq \delta_0\varepsilon$ and $\delta_0>-4aB_u^{-1}B_v$.
\end{lm}
\begin{proof}
	Firstly, we prove the inequality \eqref{relsult_lemma7_appendix1}. It is equivalent to the following inequality
	\begin{align}\label{proof_lemma7_appendix1}
		\begin{split}
			&\sqrt{\left(2B_v\beta-2B_u\alpha a-\tfrac{\Delta x}{\varepsilon}\right)^2+4(B_u\beta +B_v\alpha a+a)^2}
			\\
			<\,&\tfrac{\Delta x}{\varepsilon}-\delta_0
			+2\sqrt{\bigl(B_v\beta -B_u\alpha a-\tfrac{\delta_0}{2}\bigr)^2+(B_u\beta +B_v\alpha a+a)^2}.
		\end{split}
	\end{align}
	Since $\Delta x/\varepsilon -\delta_0>0$, the inequality \eqref{proof_lemma7_appendix1}  becomes
	\begin{align*}
		-2(B_v\beta-B_u\alpha a)+\delta_0
		<2\sqrt{\bigl(B_v\beta -B_u\alpha a-\tfrac{\delta_0}{2}\bigr)^2+(B_u\beta +B_v\alpha a+a)^2}.
	\end{align*}
	It holds for all $\alpha, \beta\in\mathbb{R}$. Secondly, we prove that
	\begin{align}\label{proof_lemma7_appendix2}
		B_v\beta +B_u\alpha a -\dfrac{\delta_0}{2}+\sqrt{\bigl(B_v\beta -B_u\alpha a-\tfrac{\delta_0}{2}\bigr)^2+(B_u\beta +B_v\alpha a+a)^2}
		<0.
	\end{align}
	It is equivalent to the following system
	\begin{align}\label{proof_lemma7_appendix3}
		\begin{cases}
			B_v\beta +B_u\alpha a -\tfrac{\delta_0}{2}<0
			\\
			\bigl(B_v\beta -B_u\alpha a-\tfrac{\delta_0}{2}\bigr)^2+(B_u\beta +B_v\alpha a+a)^2
			<\bigl(B_v\beta +B_u\alpha a -\tfrac{\delta_0}{2}\bigr)^2
		\end{cases}
	\end{align}
	We now look at the second inequality in \eqref{proof_lemma7_appendix3}
	\begin{align}\label{proof_lemma7_appendix4}
		\begin{split}
			&	\bigl(B_v\beta -B_u\alpha a-\tfrac{\delta_0}{2}\bigr)^2+(B_u\beta +B_v\alpha a+a)^2
			<\bigl(B_v\beta +B_u\alpha a -\tfrac{\delta_0}{2}\bigr)^2
			\\
			\Leftrightarrow\,
			& B_u^2\beta^2+2aB_u(1-B_v\alpha)\beta +a^2(1+B_v\alpha)^2+2\delta_0 B_u\alpha a<0.
		\end{split}
	\end{align}
	Let us define
	\begin{align*}
		\Delta_4:&=a^2B_u^2(1-B_v\alpha)^2-B_u^2\left[
		a^2(1+B_v\alpha)^2+2\delta_0 B_u\alpha a
		\right]
		\\
		&=-2aB_u^2(2aB_v+\delta_0 B_u)\alpha.
	\end{align*}
	By using the definition of $\alpha$ in Definition~\ref{Definition4}, we have $\alpha<0$.
	On the other hand, we assume that $\delta_0>-4aB_u^{-1}B_v$. Hence, we get $\Delta_4>0$. Therefore, the inequality \eqref{proof_lemma7_appendix4} holds if and only if
	\begin{align}\label{proof_lemma7_appendix6}
		\beta B_u\in\left(
		-a(1-B_v\alpha)-\sqrt{-2a(2aB_v+\delta_0 B_u)\alpha},
		-a(1-B_v\alpha)+\sqrt{-2a(2aB_v+\delta_0 B_u)\alpha}
		\right)
	\end{align}
	Now, we check that the value of $\beta$ in Definition~\ref{Definition4} satisfies \eqref{proof_lemma7_appendix6}. Firstly, we prove
	\begin{align}\label{proof_lemma7_appendix7}
		\begin{split}
			&	-a(1-B_v\alpha)+\sqrt{-2a(2aB_v+\delta_0 B_u)\alpha}<
			-a(1-B_v\alpha)-2a\sqrt{\vert B_v\alpha\vert}
			\\
			\Leftrightarrow\,
			&4a^2\vert B_v\alpha\vert <-2a(2aB_v+\delta_0B_u)\alpha.
		\end{split}
	\end{align}
	Since $B_v<0$ and the definition of $\alpha$ in Definition~\ref{Definition4}, we get $\vert B_v\alpha\vert =B_v\alpha$. So, the inequality \eqref{proof_lemma7_appendix7} becomes
	\begin{align*}
		& 4a^2 B_v\alpha <-2a(2aB_v+\delta_0B_u)\alpha
		\\
		\Leftrightarrow\,
		& \delta_0>-4aB_u^{-1}B_v.
	\end{align*}
	Therefore, the inequality \eqref{proof_lemma7_appendix7} holds.
	\\
	Secondly, we prove
	\begin{align}\label{proof_lemma7_appendix8}
		\begin{split}
			&-a(1-B_v\alpha)+2a\sqrt{\vert B_v\alpha\vert}
			<-a(1-B_v\alpha)+\sqrt{-2a(2aB_v+\delta_0 B_u)\alpha}
			\\
			\Leftrightarrow\,
			&4a^2\vert B_v\alpha\vert <-2a(2aB_v+\delta_0B_u)\alpha.
		\end{split}
	\end{align}
	Since $B_v<0$ and the property~\eqref{SAT1}, we get $\vert B_v\alpha\vert =B_v\alpha$. So, the inequality \eqref{proof_lemma7_appendix8} becomes
	\begin{align*}
		& 4a^2B_v\alpha <-2a(2aB_v+\delta_0B_u)\alpha
		\\
		\Leftrightarrow\,
		& \delta_0>-4aB_u^{-1}B_v.
	\end{align*}
	Therefore, the inequality \eqref{proof_lemma7_appendix8} holds.
	As a consequence, we can obtain the inequality \eqref{proof_lemma7_appendix4}.
	\\
	We now consider the first inequality in \eqref{proof_lemma7_appendix3}
	\begin{align*}
		B_v\beta +B_u\alpha a-\tfrac{\delta_0}{2}<0.
	\end{align*}
	It is equivalent to the following system
	\begin{align}\label{proof_lemma7_appendix13}
		\begin{cases}
			\beta <-\dfrac{aB_u}{B_v}\alpha +\tfrac{\delta_0}{2}\times \dfrac{1}{B_v}, &\text{if }B_v>0
			\\
			\beta >-\dfrac{aB_u}{B_v}\alpha +\tfrac{\delta_0}{2}\times \dfrac{1}{B_v}, &\text{if }B_v<0
		\end{cases}
	\end{align}
	We now check that the value of $\beta$ in Definition~\ref{Definition4} satisfies \eqref{proof_lemma7_appendix13}.
	Since $B_u>0$ and $B_v<0$ then we need to prove
	\begin{align}\label{proof_lemma7_appendix21}
		-\dfrac{aB_u}{B_v}\alpha+\tfrac{\delta_0}{2}\times\dfrac{1}{B_v}
		\leq  -\dfrac{aB_u(1-B_v\alpha)+2a\sqrt{B_u^2\vert B_v\alpha\vert}}{B_u^2}.
	\end{align}
	Since $B_v<0$, the inequality \eqref{proof_lemma7_appendix21} becomes
	\begin{align*}
		-2aB_v\sqrt{B_u^2\vert B_v\alpha\vert} \leq -aB_u^3\alpha+\tfrac{\delta_0}{2}B_u^2+aB_uB_v(1-B_v\alpha).
	\end{align*}
	It is equivalent to the following system
	\begin{align}\label{proof_lemma7_appendix22}
		\begin{cases}
			-aB_u^3\alpha+\tfrac{\delta_0}{2}B_u^2+aB_uB_v(1-B_v\alpha)>0
			\\
			4a^2B_u^2B_v^2\vert B_v\alpha\vert \leq  \left[-aB_u^3\alpha+\tfrac{\delta_0}{2}B_u^2+aB_uB_v(1-B_v\alpha)\right]^2
		\end{cases}
	\end{align}
	We first look at the second inequality in \eqref{proof_lemma7_appendix22}
	\begin{align}\label{proof_lemma7_appendix23}
		4a^2B_u^2B_v^2\vert B_v\alpha\vert \leq \left[-aB_u^3\alpha+\tfrac{\delta_0}{2}B_u^2+aB_uB_v(1-B_v\alpha)\right]^2.
	\end{align}
	Since $B_v<0$ and $\alpha$ satisfies~\eqref{SAT1}, the inequality~\eqref{proof_lemma7_appendix23} can be reformulated as
	\begin{align}\label{proof_lemma7_appendix25}
		\begin{split}
			&4a^2B_u^2B_v^3\alpha \leq \left[-aB_u^3\alpha+\tfrac{\delta_0}{2}B_u^2+aB_uB_v(1-B_v\alpha)\right]^2
			\\
			\Leftrightarrow\,
			& a^2(B_u^2+B_v^2)^2\alpha^2-\left(
			6a^2B_v^3 + a\delta_0B_u^3 + 2a^2B_u^2B_v + a\delta_0 B_uB_v^2
			\right)\alpha +\left(\tfrac{\delta_0}{2}B_u+a B_v\right)^2 \geq 0.
		\end{split}
	\end{align}
	On the other hand, since $B_v<0$ and under the assumption $\delta_0>-4aB_u^{-1}B_v$, the inequality~\eqref{proof_lemma7_appendix25} holds for all $\alpha\in\mathbb{R}$.
	\\
	Then, we consider  the first inequality in \eqref{proof_lemma7_appendix22}
	\begin{align}\label{proof_lemma7_appendix26}
		-aB_u^3\alpha+\tfrac{\delta_0}{2}B_u^2+aB_uB_v(1-B_v\alpha)>0
	\end{align}
	Since $B_u>0$, the inequality \eqref{proof_lemma7_appendix26} can be represented as
	\begin{align*}
		\alpha<\dfrac{\tfrac{\delta_0}{2}B_u+aB_v}{a(B_u^2+B_v^2)}.
	\end{align*}
	Under the assumption $\delta_0>-4aB_u^{-1}B_v$, we get
	\begin{align*}
		\tfrac{\delta_0}{2}B_u+aB_v>0.
	\end{align*}
	Therefore, the inequality \eqref{proof_lemma7_appendix26} holds under the condition on $\alpha$ in \eqref{SAT1}.
	As a consequence, we get the first inequality in \eqref{proof_lemma7_appendix3}. 
	This ends the proof of Lemma \ref{lemma7_appendix}.\hspace{3cm}
\end{proof}


\section*{Acknowledgement}
Research of Nguyen Thi Hoai Thuong was partially supported by  Vietnam National University, Ho Chi Minh City (VNU-HCM) under grant number C2022-18-47.\\
Research of Benjamin Boutin was partially supported by ANR project HEAD, ANR-24-CE40-3260 and by Centre Henri Lebesgue, programme ANR-11-LABX-0020-01.


\bibliographystyle{abbrv}
\bibliography{NguyenBoutinBiblio}

\begin{thebibliography}{10}

\bibitem{Benzoni-GavageSerre07}
S.~{Benzoni-Gavage} and D.~Serre.
\newblock {\em Multidimensional hyperbolic partial differential equations:
  first-order systems and applications}.
\newblock Oxford mathematical monographs. Clarendon Press, Oxford ; New York,
  2007.

\bibitem{BoutinNguyenSeguin20}
B.~Boutin, T.~H.~T. Nguyen, and N.~Seguin.
\newblock A stiffly stable semi-discrete scheme for the characteristic linear
  hyperbolic relaxation with boundary.
\newblock {\em ESAIM: Mathematical Modelling and Numerical Analysis},
  54:1569--1596, July 2020.

\bibitem{CaoYong22}
X.~Cao and W.-A. Yong.
\newblock Construction of boundary conditions for hyperbolic relaxation
  approximations {{II}}: {{Jin-Xin}} relaxation model.
\newblock {\em Quarterly of Applied Mathematics}, 80:787--816, 2022.

\bibitem{CarpenterGottliebAbarbanel94}
M.~H. Carpenter, D.~Gottlieb, and S.~Abarbanel.
\newblock Time-{{Stable Boundary Conditions}} for {{Finite-Difference Schemes
  Solving Hyperbolic Systems}}: {{Methodology}} and {{Application}} to
  {{High-Order Compact Schemes}}.
\newblock {\em Journal of Computational Physics}, 111:220--236, Apr. 1994.

\bibitem{CarpenterNordstromGottlieb99}
M.~H. Carpenter, J.~Nordström, and D.~Gottlieb.
\newblock A stable and conservative interface treatment of arbitrary spatial
  accuracy.
\newblock {\em Journal of Computational Physics}, 148:341--365, 1999.

\bibitem{Clarke78}
J.~F. Clarke.
\newblock Gas dynamics with relaxation effects.
\newblock {\em Reports on Progress in Physics}, 41:807--864, June 1978.

\bibitem{ColellaMajdaRoytburd86}
P.~Colella, A.~Majda, and V.~Roytburd.
\newblock Theoretical and numerical structure for reacting shock waves.
\newblock {\em Society for Industrial and Applied Mathematics. Journal on
  Scientific and Statistical Computing}, 7:1059--1080, 1986.

\bibitem{CoquelJinLiuEtAl14}
F.~Coquel, S.~Jin, J.-G. Liu, and L.~Wang.
\newblock Well-posedness and singular limit of a semilinear hyperbolic
  relaxation system with a two-scale discontinuous relaxation rate.
\newblock {\em Archive for Rational Mechanics and Analysis}, 214:1051--1084,
  Dec. 2014.

\bibitem{GustafssonKreissOliger13}
B.~Gustafsson, H.-O. Kreiss, and J.~Oliger.
\newblock {\em Time-dependent problems and difference methods}.
\newblock Pure and {{Applied Mathematics}} ({{Hoboken}}). John Wiley \& Sons,
  Inc., Hoboken, NJ, 2 edition, 2013.

\bibitem{HanouzetNatalini03}
B.~Hanouzet and R.~Natalini.
\newblock Global existence of smooth solutions for partially dissipative
  hyperbolic systems with a convex entropy.
\newblock {\em Archive for Rational Mechanics and Analysis}, 169:89--117, 2003.

\bibitem{Hindmarsh83}
A.~C. Hindmarsh.
\newblock {{ODEPACK}}, a systematized collection of {{ODE}} solvers.
\newblock In {\em Scientific computing ({{Montreal}}, {{Que}}., 1982)},
  volume~I of {\em {{IMACS Trans}}. {{Sci}}. {{Comput}}.}, pages 55--64. IMACS,
  New Brunswick, NJ, 1983.

\bibitem{HuangLiZhou23}
J.~Huang, R.~Li, and Y.~Zhou.
\newblock Coupling conditions for linear hyperbolic relaxation systems in
  two-scale problems.
\newblock {\em Mathematics of Computation}, 92:2133--2165, Sept. 2023.

\bibitem{InglardLagoutiereRugh20}
M.~Inglard, F.~Lagoutière, and H.~H. Rugh.
\newblock Ghost solutions with centered schemes for one-dimensional transport
  equations with {{Neumann}} boundary conditions.
\newblock {\em Annales de la Faculté des Sciences de Toulouse. Mathématiques.
  Série 6}, 29:927--950, 2020.

\bibitem{JinXin95}
S.~Jin and Z.~Xin.
\newblock The relaxation schemes for systems of conservation laws in arbitrary
  space dimensions.
\newblock {\em Communications on Pure and Applied Mathematics}, 48:235--276,
  1995.

\bibitem{KreissScherer74}
H.-O. Kreiss and G.~Scherer.
\newblock Finite element and finite difference methods for hyperbolic partial
  differential equations.
\newblock {\em Mathematical Aspects of Finite Elements in Partial Differential
  Equations}, Dec. 1974.

\bibitem{LiuYong01}
H.~Liu and W.-A. Yong.
\newblock Time-asymptotic stability of boundary-layers for a hyperbolic
  relaxation system.
\newblock {\em Communications in Partial Differential Equations},
  26:1323--1343, June 2001.

\bibitem{Liu87}
T.-P. Liu.
\newblock Hyperbolic conservation laws with relaxation.
\newblock {\em Communications in Mathematical Physics}, 108:153--175, 1987.

\bibitem{Petzold83}
L.~Petzold.
\newblock Automatic selection of methods for solving stiff and nonstiff systems
  of ordinary differential equations.
\newblock {\em Society for Industrial and Applied Mathematics. Journal on
  Scientific and Statistical Computing}, 4:137--148, 1983.

\bibitem{Stoker92}
J.~J. Stoker.
\newblock {\em Water waves. the mathematical theory with applications. reprint
  of the 1957 original.}
\newblock New York, NY: Wiley, reprint of the 1957 original edition, 1992.

\bibitem{Strand94}
B.~Strand.
\newblock Summation by parts for finite difference approximations for d/dx.
\newblock {\em Journal of Computational Physics}, 110:47--67, 1994.

\bibitem{Trefethen84}
L.~N. Trefethen.
\newblock Instability of difference models for hyperbolic initial-boundary
  value problems.
\newblock {\em Communications on Pure and Applied Mathematics}, 37:329--367,
  1984.

\bibitem{Whitham74}
G.~B. Whitham.
\newblock {\em Linear and nonlinear waves.}
\newblock John Wiley \& Sons, Hoboken, NJ, 1974.

\bibitem{XinXu00}
Z.~Xin and W.-Q. Xu.
\newblock Stiff well-posedness and asymptotic convergence for a class of linear
  relaxation systems in a quarter plane.
\newblock {\em Journal of Differential Equations}, 167:388--437, Nov. 2000.

\bibitem{Yong99}
W.-A. Yong.
\newblock Singular perturbations of first-order hyperbolic systems with stiff
  source terms.
\newblock {\em Journal of Differential Equations}, 155:89--132, 1999.

\bibitem{Yong04}
W.-A. Yong.
\newblock Entropy and global existence for hyperbolic balance laws.
\newblock {\em Archive for Rational Mechanics and Analysis}, 172:247--266, May
  2004.

\bibitem{YongZhou21}
W.-A. Yong and Y.~Zhou.
\newblock Recent {{Advances}} on {{Boundary Conditions}} for {{Equations}} in
  {{Nonequilibrium Thermodynamics}}.
\newblock {\em Symmetry}, 13:1710, Sept. 2021.

\bibitem{ZhaoHuangYong19}
W.~Zhao, J.~Huang, and W.-A. Yong.
\newblock Boundary conditions for kinetic theory based models {{I}}: {{Lattice
  Boltzmann}} models.
\newblock {\em Multiscale Modeling \& Simulation. A SIAM Interdisciplinary
  Journal}, 17:854--872, 2019.

\bibitem{ZhaoYong21}
W.~Zhao and W.-A. Yong.
\newblock Boundary conditions for kinetic theory-based models {{II}}: {{A}}
  linearized moment system.
\newblock {\em Mathematical Methods in the Applied Sciences}, 44:14148--14172,
  2021.

\bibitem{ZhouYong21}
Y.~Zhou and W.-A. Yong.
\newblock Boundary conditions for hyperbolic relaxation systems with
  characteristic boundaries of type {{I}}.
\newblock {\em Journal of Differential Equations}, 281:289--332, 2021.

\bibitem{ZhouYong22}
Y.~Zhou and W.-A. Yong.
\newblock Boundary conditions for hyperbolic relaxation systems with
  characteristic boundaries of type {{II}}.
\newblock {\em Journal of Differential Equations}, 310:198--234, 2022.

\end{thebibliography}

\end{document}